\documentclass[a4paper]{amsart}   
\usepackage{graphicx}
\usepackage{caption}
\usepackage{amssymb}
\usepackage{subcaption}
\usepackage{amstext}
\usepackage{amsmath}
\usepackage{amsthm}

\def\tilL{\(\widetilde{\mathcal{L}}\)}
\def\R3{\(\mathbb{R}^{3}\)}

\newtheorem{definition}{Definition}

\newtheorem{theorem}{Theorem}
\newtheorem{proposition}{Proposition}
\newtheorem{corollary}{Corollary}

\newtheorem{rem}{Remark}

\begin{document}

\title[Projected Symmetries]{Hexagonal Projected Symmetries}

\author[J.F. Oliveira, S.B.S.D.Castro, I.S.Labouriau]{Juliane F. Oliveira, Sofia B.S.D. Castro, Isabel S. Labouriau}
\address{Centro de Matem\'atica da Universidade do Porto, Rua do Campo Alegre 687, 4169-007, Porto, Portugal \\ and
J.F. Oliveira, I.S.Labouriau --- Departamento de Matem\'atica, Faculdade de Ci\^encias da Universidade do Porto,Portugal\\
S.B.S.D.Castro --- Faculdade de Economia, Universidade do Porto, Rua Dr. Roberto
Frias, 4200-464, Porto, Portugal}

\maketitle                        

\begin{abstract}
		In the study of pattern formation in symmetric physical systems a 3-dimensional structure in thin domains is often modelled as 2-dimensional one. We are concerned with functions in \(\mathbb{R}^{3}\) that are invariant under the action of a crystallographic group and the symmetries of their projections into a function defined on a plane. We obtain a list of the crystallographic groups for which the projected functions have a hexagonal lattice of periods. The proof is constructive and the result may be used in the study of observed patterns in thin domains, whose symmetries are not expected in 2-dimensional models, like the black-eye pattern.
\end{abstract}

\section{Introduction}

	This article is aimed at two sets of readers: crystallographers and bifurcation theorists. What is obvious to one set of readers is not necessarily so to the other. 
Concepts are also subject to alternative, substantially different, statements.
We try to provide a bridge between the two viewpoints, conditioned by our background in bifurcations.

	In the study of crystals and quasicrystals, projection is a mathematical tool for lowering dimension, \cite{senechal96}, \cite{Koca14}. A well developed study in crystallographic groups, their subgroups and  the notion of projection used in crystallography can be found in the International Tables for Crystallography (ITC) volume A \cite{int-tab-a} and ITC volume E \cite{int-tab-e}. Tables therein provide information on projections of elements of crystallographic groups. 
	
	However, we intend to use crystallographic groups for a different purpose. 
The symmetries of solutions of partial differential equations, under certain boundary conditions, form a crystallographic group --- see, for instance Golubitsky and Stewart \cite[chapter 5]{goluste02}. The set of all level curves of these functions is interpreted as a pattern. 
	In order to study 3-dimensional patterns observed in a 2-dimensional environment, we use the projection of  
symmetric
 functions as defined in section 2.  The symmetry group of the projected functions does not necessarily coincide with that of projections used in crystallography. 
The  information contained in the  ITC  \cite{int-tab-a,int-tab-e},  has to be organised in a different way before it can be used for this purpose, and this is the object of the present article.

	Regular patterns are usually seen directly in nature and experiments. Convection, reaction-diffusion systems and the Faraday waves experiment comprise three commonly studied pattern-forming systems, see for instance \cite{busse78}, \cite{turing52}, \cite{crawgola93}. 
	
	Equivariant bifurcation  theory has been used extensively to study pattern formation via symmetry-breaking steady state bifurcation in various physical systems modelled by  \(E(n)\)-equivariant partial differential equations. In Golubitsky and Stewart \cite[chapter 5] {goluste02} there is a complete description of this method used, for example, in  Dionne and Golubitsky \cite{digolu92}, Dionne \cite{dionne93}, Bosch Vivancos et al. \cite{vichome95}, Callahan and Knobloch \cite{calkno97} and Dionne et al. \cite{disilske97}, where the spatially periodic patterns are sometimes called planforms.
	
	The pattern itself and its observed state can occur in different dimensions. This happens for instance when an experiment is done in a 3-dimensional medium but the patterns are only observed on a surface, a 2-dimensional object. This is the case for reaction-diffusion systems in the Turing instability regime, \cite{turing52}, 
which have often been described using a 2-dimensional representation \cite{ouswi95}. The interpretation of this 2-dimensional outcome is subject to discussion: the black-eye pattern observed by \cite{ouswi95} has been explained both as a mode interaction, \cite{gunaouswi94}, and as a suitable projection of a 3-dimensional into a 2-dimensional lattice \cite{gomes99}. In her article, Gomes shows how a 2-dimensional hexagonal pattern can be produced by a specific projection of a Body Centre Cubic (bcc) lattice. 	
									
	Pinho and Labouriau \cite{pinlab14} study projections in order to understand how these affect symmetry. Their necessary and sufficient conditions for identifying projected symmetries are used extensively in our results. 
	
	Motivated by the explanation of Gomes \cite{gomes99} we look for all 3-dimensional lattices that exhibit a hexagonal projected lattice. We illustrate our results using the Primitive Cubic lattice.	

   \section{Projected Symmetries}
   
	 The study of projections is related to patterns. Patterns are level curves of functions \(f: \mathbb{R}^{n+1}\rightarrow\mathbb{R}\). In our work we suppose that these functions are invariant under the action of a particular subgroup of the Euclidean group: a crystallographic group.

	The Euclidean group, \(E(n+1)\), is the group of all isometries on \(\mathbb{R}^{n+1}\), also described by the semi-direct sum \(E(n+1)\cong \mathbb{R}^{n+1}\dotplus O(n+1)\), with elements given as an ordered pair \((v,\delta)\), in which \(v \in \mathbb{R}^{n+1}\) and \(\delta\) is an element of the orthogonal group \(O(n+1)\) of dimension \(n+1\). 

	Let \(\Gamma\) be a subgroup of \(E(n+1)\). The homomorphism 
\[\begin{array}{cccc}
\phi: & \Gamma & \rightarrow & O(n+1)\\
      & (v,\delta) & \mapsto & \delta
\end{array}\]
has as image a group \(\mathbf{J}\), called the \emph{point group} of \(\Gamma\), and its kernel forms the \emph{translation subgroup} of \(\Gamma\). 

	We say that the translation subgroup of \(\Gamma\) is an \emph{\((n+1)\)-dimensional lattice}, \( \mathcal{L}\) , if it is generated over the integers by \(n+1\) linearly independent elements \(l_{1},\cdots,l_{n+1} \in \mathbb{R}^{n+1}\), which we write:
\[\mathcal{L} = \langle l_{1},\cdots,l_{n+1}\rangle_{\mathbb{Z}}\] 

	A \emph{crystallographic group} is a subgroup of \(E(n+1)\), such that its translation subgroup is an \((n+1)\)-dimensional lattice.
	
A description of these concepts can be found in the ITC volume A  \cite{int-tab-a}, chapter 8.1, pp. 720-725 and in their suggested bibliography for the chapter;  see also \cite{miller72} .
	
	To get symmetries of objects in \(\mathbb{R}^{n+1}\), consider the group \emph{action} of \(E(n+1)\) on \(\mathbb{R}^{n+1}\) given by the function:
\begin{equation} \label{action}
\begin{array}{ccc}
E(n+1)\times\mathbb{R}^{n+1} & \rightarrow & \mathbb{R}^{n+1} \\
((v,\delta), (x,y)) & \mapsto & (v,\delta)\cdot(x,y) = v + \delta(x,y)
\end{array}
\end{equation}

	In \cite{armstrong88}, the reader can see that the action \eqref{action} restricted to a point group of a crystallographic group leaves its translation subgroup \( \mathcal{L}\) \ invariant. The largest subgroup of \(O(n+1)\) that leaves \( \mathcal{L}\) \ invariant forms the \emph{holohedry} of \( \mathcal{L}\) \ and is denoted by \(H_{\mathcal{L}}\). The holohedry is always a finite group, see \cite{senechal96}, subsection 2.4.2. 
Note that the term \textit{holohedry} used here, as well as in \cite{digolu92} and  \cite{goluste02}, corresponds in \cite{int-tab-a}, chapter 8.2, to the definition of point symmetry of the lattice.

	Crystallographic groups are related to symmetries of pattern formation by the action of the group of symmetries on a space of functions, 
\cite{goluste02} chapter 5.
To see this, observe that \eqref{action} induces an action of a crystallographic group \(\Gamma\) on the space of functions \(f:\mathbb{R}^{n+1}\rightarrow\mathbb{R}\) by:
	
\[(\gamma\cdot f)(x,y) = f(\gamma^{-1}(x,y)) \ \mbox{for} \ \gamma \in \Gamma \ \mbox{and} \ (x,y) \in \mathbb{R}^{n+1}\]

	Thus, we can construct a space \(\mathcal{X}_{\Gamma}\) of \(\Gamma\)-invariant functions, that is
\[\mathcal{X}_{\Gamma} = \lbrace f:\mathbb{R}^{n+1}\rightarrow\mathbb{R}; \ \gamma\cdot f = f, \ \forall \gamma \in \Gamma \rbrace\] 
	
	In particular a \(\Gamma\)-invariant function is \( \mathcal{L}\) -invariant. 
	
	An \emph{ \( \mathcal{L}\) -symmetric pattern} or \emph{ \( \mathcal{L}\) -crystal pattern} consists of the 
	 set of all
 level curves of a function \(f:\mathbb{R}^{n+1}\rightarrow\mathbb{R}\) with periods in the lattice  \( \mathcal{L}\) .	
	
	In \cite{gomes99} the black-eye pattern is obtained as a projection of a function, whose level sets form a bcc pattern in \(\mathbb{R}^{3}\). In terms of symmetries, the black-eye is a hexagonal pattern, as we can see in \cite{gomes99}, it is the level sets of a  bidimensional function with periods in a hexagonal plane lattice, that is, a lattice that admits as its holohedry a group isomorphic to the dihedral group of symmetries of the regular hexagon, \(D_{6}\). 
Moreover, we expect  the point group of symmetries of the black-eye to be isomorphic  to \(D_{6}\).

	For \(y_{0} > 0\), consider the restriction of \(f \in \mathcal{X}_{\Gamma}\) to the region between the hyperplanes \(y=0\) and \(y = y_{0}\). The projection operator \(\Pi_{y_{0}}\) integrates this restriction of \(f\) along the width \(y_{0}\), yielding a new function with domain \(\mathbb{R}^{n}\).
	
\begin{definition}\label{defProjection}
	For \(f \in \mathcal{X}_{\Gamma}\) and \(y_{0} > 0\), the \textbf{projection operator} \(\Pi_{y_{0}}\) is given by:
	\[\Pi_{y_{0}}(f)(x) = \displaystyle \int _{0}^{y_{0}}f(x,y)dy\]
\end{definition}	

	The region between \(y = 0\) and \(y = y_{0}\) is called the \textbf{projection band} and \(\Pi_{y_{0}}(f):\mathbb{R}^{n}\rightarrow\mathbb{R}\) is the \textbf{projected function}.
	
	The functions \(\Pi_{y_{0}}(f)\) may be invariant under the action of some elements of the group \(E(n)\cong \mathbb{R}^{n}\dotplus O(n)\). The relation between the symmetries of \(f\) and those  of \(\Pi_{y_{0}}(f)\) was provided by Pinho and Labouriau \cite{pinlab14}. 
	
	To find the group of symmetries of the projected functions \(\Pi_{y_{0}}(\mathcal{X}_{\Gamma})\), the authors consider the following data: 

\begin{itemize}
\item for \(\alpha\in O(n)\), the elements of $O(n+1)$:
\[\sigma := \left(\begin{array}{cc}
I_{n} & 0 \\ 
0 & -1
\end{array}\right), \ \alpha_{+} := \left(\begin{array}{cc}
\alpha & 0 \\ 
0 & 1
\end{array}\right) \ \mbox{and} \ \  \alpha_{-} := \sigma\alpha_{+};\]

\item the subgroup \(\widehat{\Gamma}\) of \(\Gamma\), whose elements are of the form
\[\left((v,y),\alpha_\pm\right);\ \alpha \in O(n), \ (v,y) \in \mathbb{R}^{n};\times\mathbb{R} \]
 the translation subgroup of \(\widehat{\Gamma}\) and \(\Gamma\) are the same, while the point group of \(\widehat{\Gamma}\) consists of those elements of \(\Gamma\) that fix the space \(\lbrace (v,0) \in \mathbb{R}^{n}\rbrace\).

	For the  3-dimensional case  \(\widehat{\Gamma}\) coincides with the  \emph{scanning group} defined in \cite{int-tab-e}, chapter 5.2.

\item and the projection \(h:\widehat{\Gamma}\rightarrow E(n)\cong\mathbb{R}^{n}\dotplus O(n)\) given by:
\[h\left( (v,y),\alpha_\pm \right)=(v,\alpha)\]

\end{itemize}

	The group of symmetries of \(\Pi_{y_{0}}(\mathcal{X}_{\Gamma})\) is the image by the projection \(h\) of the group \(\Gamma_{y_{0}}\) defined as: 
\begin{itemize}
\item If \((0,y_{0})\in \mathcal{L}\) then \(\Gamma_{y_{0}}=\widehat{\Gamma}\).
\item If \((0,y_{0})\notin \mathcal{L}\) then \(\Gamma_{y_{0}}\) contains only those elements of \(\widehat{\Gamma}\) that are either 
side preserving
\(\left((v,0),\alpha_{+}\right)\) or 
side reversing
\(\left((v,y_{0}),\alpha_{-}\right)\)
\end{itemize}

	The group \(\widehat{\Gamma}\) consists of those elements of \(\Gamma\) that will contribute to the symmetries of the set of projected functions.  Depending on  whether the hypotheses above hold, the group  \(\Gamma_{y_{0}}\)  will be either the whole group \(\widehat{\Gamma}\) or a subperiodic group of \(\widehat{\Gamma}\), that is a subgroup whose lattice of translations has lower dimension than the space on which the group acts; see \cite{int-tab-a} chapter 8.1 and \cite{int-tab-e} chapter 1.2. 

	The group \(\Gamma_{y_{0}}\) depends on how the elements of \(\Gamma\) are transformed by the projection 	\(\Pi_{y_{0}}\)  of \(\Gamma\)-invariant functions. The criterion that clarifies the connection between the symmetries of \(\mathcal{X}_{\Gamma}\) and \(\Pi_{y_{0}}(\mathcal{X}_{\Gamma})\) is provided by the following result:
	
\begin{theorem}[Theorem 1.2 in \cite{pinlab14}]\label{thPinho}
	All functions in \(\Pi_{y_{0}}(\mathcal{X}_{\Gamma})\) are invariant under the action of \((v,\alpha)\in E(n)\) if and only if one of the following conditions holds:
\begin{itemize}
\item[I]$((v,0),\alpha_{+})\in \Gamma$;
\item[II]$((v,y_{0}),\alpha_{-})\in \Gamma$;
\item[III]$(0,y_{0})\in \mathcal{L}$ and either $((v,y_{1}),\alpha_{+})\in \Gamma$ or $((v,y_{1}),\alpha_{-})\in \Gamma$, for some $y_{1}\in\mathbb{R}$.
\end{itemize}
\end{theorem}

	\section{Hexagonal Projected Symmetries}

	As we saw in the last section, there is a connection between a crystallographic group \(\Gamma\) in dimension \(n+1\) and the group of symmetries of the set of projected functions \(\Pi_{y_{0}}(\mathcal{X}_{\Gamma})\). In this work we aim to know which crystallographic groups in dimension 3 can yield hexagonal symmetries after projection. In other words, we want to describe  how to obtain hexagonal plane patterns by projection.   
	
		
	Given a crystallographic group \(\Gamma\), with an \( (n+1)\)-dimensional lattice  \( \mathcal{L}\) , whose holohedry is $H_{\mathcal{L}}$, we denote by \(\Pi_{y_{0}}(\mathcal{L})\) the translation subgroup of the crystallographic group \(\Pi_{y_{0}}(\Gamma)\) of symmetries of \(\Pi_{y_{0}}(\mathcal{X}_{\Gamma})\), whose point group is a subset of the holohedry of $\Pi_{y_{0}}(\mathcal{L})$. From theorem~\ref{thPinho} we obtain
	
\begin{corollary}\label{corolHolohedry1}
	Let \(\widetilde{\Gamma}\) be a crystallographic group with lattice \tilL \ $\subset \mathbb{R}^{n}$. Suppose  \tilL \ $=\Pi_{y_{0}}(\mathcal{L})$, and let $H_{\widetilde{\mathcal{L}}}$ and $H_{\mathcal{L}}$ be the holohedries of \tilL \ and $\mathcal{L}\subset \mathbb{R}^{n+1}$, respectively.  If  $\alpha \in H_{\widetilde{\mathcal{L}}}$ lies in the point group of  \(\widetilde{\Gamma}\) then either $\alpha_{+} \in H_{\mathcal{L}}$ or $\alpha_{-} \in H_{\mathcal{L}}$.
\end{corollary}	

\begin{proof}
	Since \(\alpha \in H_{\widetilde{\mathcal{L}}}\) implies \(\alpha\) lies in the point group of \(\Gamma\), then there exists  \(v \in \mathbb{R}^{n}\) such that \(f\) is \((v,\alpha)\)-invariant for all \(f \in \Pi_{y_{0}}(\mathcal{X}_{\Gamma})\). Hence, one of the three conditions of theorem~\ref{thPinho}  holds. Then, depending on whether (I), (II) or (III) is verified, either \((w,\alpha_{+})\) or \((w,\alpha_{-})\) is in \(\Gamma\) where, \(w \in \lbrace (v,0), \ (v,y_{0}), \ (v,y_{1})\rbrace\). By definition of holohedry, we have either  \(\alpha_{+} \in H_{\mathcal{L}}$ or $\alpha_{-} \in H_{\mathcal{L}}\).
%
\end{proof}		

\begin{rem}\label{remProjLattice}
	\emph{We note that there is a non-trivial relation between the lattice \(\widetilde{\mathcal{L}}\) of periods of the projected functions and that of the original one. In fact, consider an \((n+1)\)-dimensional lattice  \( \mathcal{L}\) \ and \(\widetilde{\mathcal{L}} = \Pi_{y_{0}}(\mathcal{L})\). If \(v \in \widetilde{\mathcal{L}}\) then \((v,I_{n})\) is a symmetry of \(\Pi_{y_{0}}(\mathcal{X}_{\Gamma})\). Applying theorem~\ref{thPinho} with \(\alpha = I_{n}\), one of the following holds for each \(v \in \widetilde{\mathcal{L}}\):
\begin{itemize}
\item[I]\(((v,0),\alpha_{+}) = ((v,0),I_{n+1})\in \Gamma\), or equivalently  \((v,0) \in \mathcal{L}\);
\item[II]\(((v,y_{0}),\alpha_{-}) = ((v,y_{0}),\sigma) \in \Gamma\) then \(((v,y_{0}),\sigma)^{2}\in \Gamma\) implying that \((2v,0)\in \mathcal{L}\);
\item[III]\((0,y_{0})\in \mathcal{L}$ and either \((v,y_{1})\) or \((2v,0)\) is in  \( \mathcal{L}\) , for some $y_{1}\in\mathbb{R}$.
\end{itemize}}

	\emph{While condition I implies that \(\mathcal{L}\cap\lbrace (x,0)\in \mathbb{R}^{n+1}\rbrace\subseteq\widetilde{\mathcal{L}}\), the other conditions show that this inclusion is often strict. Furthermore, conditions II and III show that we may have no element of the form \((v,y_{1})\) in  \( \mathcal{L}\) \ and yet \(v\in\widetilde{\mathcal{L}}\). This is due to a possible non-zero 
	translation
vector associated to \(\sigma \in J\). }
\end{rem}

	As a converse to corollary~\ref{corolHolohedry1} we have:

\begin{corollary}\label{corolHolohedry2}
Let $\Gamma$ be a crystallographic group with an $(n+1)$-dimensional lattice  \( \mathcal{L}\){} and let  \tilL \ $=\Pi_{y_{0}}(\mathcal{L})$, with $H_{\widetilde{\mathcal{L}}}$ and $H_{\mathcal{L}}$  the holohedries of 
$\widetilde{\mathcal{L}}\subset \mathbb{R}^{n}$
and of $\mathcal{L}\subset \mathbb{R}^{n+1}$, respectively. Suppose either $\alpha_{+} $ or $ \alpha_{-}$ is in $H_{\mathcal{L}}$. If one of the following conditions holds
\begin{itemize}
\item[1.] $\sigma \notin H_{\mathcal{L}}$;
\item[2.] either $(0, \alpha_{+}) $ or $(0, \alpha_{-})$  is in $\Gamma$;
\end{itemize}	
then, for any $y_{0} \in \mathbb{R}$, $\alpha \in H_{\widetilde{\mathcal{L}}}$.
\end{corollary}	

\begin{proof}
Consider $v \in \widetilde{\mathcal{L}}$ and suppose condition 1 holds. By theorem~\ref{thPinho}, either $(v,0) \in \mathcal{L}$ or  $(0,y_{0})$ and  $(v,y_{1}) \in \mathcal{L}$.
	
If $(v,0) \in \mathcal{L}$ then $(\alpha v,0) \in \mathcal{L}$. Otherwise, $(0,y_{0})$ and  $(\alpha v,y_{2}) \in \mathcal{L}$  for $y_{2} \in \lbrace -y_{1}, \ y_{1}\rbrace$ . Applying theorem~\ref{thPinho} we have $\alpha v \in \widetilde{\mathcal{L}}$ in both cases. Therefore, $\alpha$ is a symmetry of $\widetilde{\mathcal{L}}$.
	
	Suppose now that condition  2 holds. If  $(0, \alpha_{+}) \in \Gamma$ then $(0, \alpha)$ belongs to  \(\widetilde{\Gamma}\), for all $y_{0} \in \mathbb{R}$, 
by condition I of theorem~\ref{thPinho}.
The other possibility is that $(0, \alpha_{-}) \in \Gamma$. 
 If for $v \in \widetilde{\mathcal{L}}$ either condition I or condition III of theorem~\ref{thPinho} holds,  the proof follows as in the case of condition 1. Suppose then that $((v,y_{0}), \sigma) \in \Gamma$, then $ ((v,y_{0}), \sigma)\cdot(0, \alpha_{-}) = ((v,y_{0}), \alpha_{+}) \in \Gamma$.  Therefore, $(v, \alpha) \in \widetilde{\Gamma}$, 
  by theorem~\ref{thPinho},
  completing the proof. 
%
\end{proof}		

	The analysis in \cite{int-tab-e} chapter 5.1 aims to find \emph{sectional layer groups} and \emph{penetration rod groups}, subgroups of the crystallographic group that leave a \emph{crystallographic plane}, defined by three lattice points, and a \emph{crystallographic straight line} invariant, respectively, by a method of scanning a given crystallographic group.
	
	When a pattern is projected, it is not immediate that the plane of projection is crystallographic, as it may not contain three lattice points.

We say that lattice $ \mathcal{L}_1$ is \emph{rationally compatible} with a lattice $ \mathcal{L}_2$ if there exists \(r \in \mathbb{Z} \setminus\lbrace 0 \rbrace \) such that 
$r \mathcal{L}_1\subset \mathcal{L}_2$.
A vector $v\in\mathbb{R}^{k}$ is \emph{rational with respect to a lattice} 
$ \mathcal{L}\subset\mathbb{R}^{k}$ if 
$\langle v,\ell\rangle\in\mathbb{Q}$ for all $\ell\in \mathcal{L}$,
where \(\langle\cdot,\cdot\rangle\) is the usual inner product in \(\mathbb{R}^{k}\). 
Given a lattice \(\widetilde{\mathcal{L}}\subset\mathbb{R}^{k}\), we define its \emph{suspension}
 \(\widetilde{\mathcal{L}}_s\subset \mathbb{R}^{k+1}\) as 
  \(\widetilde{\mathcal{L}}_s=\{(v,0); \ v \in \widetilde{\mathcal{L}}\}\).

The next proposition provides  conditions  for a suspension of the projected lattice to be rationally compatible with the original lattice.

\begin{proposition}\label{propRational}
	Consider a crystallographic group \(\Gamma\) with a lattice \(\mathcal{L} \subset \mathbb{R}^{n+1}\) and let \(\widetilde{\mathcal{L}} = \Pi_{y_{0}}(\mathcal{L}) \subset \mathbb{R}^{n} \) be the translation subgroup of \(\Pi_{y_{0}}({\Gamma})\) and denote its suspension by
  \(\widetilde{\mathcal{L}}_s\subset \mathbb{R}^{n+1}\). \\
 If  \((0,y_{0})\not\in \mathcal{L}\), then  \(\widetilde{\mathcal{L}}_s \) is always rationally compatible with  \(\mathcal{L}\).\\
  If \((0,y_{0})\in \mathcal{L}\), then  \(\widetilde{\mathcal{L}}_s\) is rationally compatible with \(\mathcal{L}\) if and only if
 the normal vector  \((0,y_{0})\)  to the projection hyperplane is rational with respect to \( \mathcal{L}\).
\end{proposition}

Note that if  \((0,y_{0})\in \mathcal{L}\), we are projecting the values of functions on a band of the width of one (or more) cells along a crystallographic direction. Otherwise the projected group is smaller.
So, we must use different results in the ITC according to the case.

\begin{proof}	
  If  \((0,y_{0})\not\in \mathcal{L}\), then only conditions I or II of remark~\ref{remProjLattice} are applicable.
Therefore,  if  \( v\in \widetilde{\mathcal{L}}\) then \( (2v,0)\in\mathcal{L}\). 
Hence,  \(\widetilde{\mathcal{L}} \)   is  rationally compatible with  \(\mathcal{L}\).

 If \((0,y_{0})\in \mathcal{L}\), then, using remark~\ref{remProjLattice}, it follows that  \( v\in \widetilde{\mathcal{L}}\) for all \( v\) such that  \((v,y_{1}) \in \mathcal{L}\) for some \(y_{1}\in\mathbb{R}\). 
 
Suppose first that $  \widetilde{\mathcal{L}}_s$ is rationally compatible with $\mathcal{L}$ and 
let \( (v,y_{1}) \in \mathcal{L}\) .
 Then \( r(v,0)\in\mathcal{L}\) and  hence  \((0,ry_1) \in \mathcal{L}\) for some \(r\in\mathbb{Z}\).
It follows that \(y_1=\frac{p}{q}y_{0}\) for some \(p, \ q\) nonzero integers. Therefore   \((0,y_{0})\)  is rational with respect to \( \mathcal{L}\).

  Now suppose  \((0,y_{0})\)  is rational with respect to \( \mathcal{L}\).
We claim that for each \(\widetilde{l_{{j}}}\) one of the following conditions holds:
\begin{itemize}
	\item[(i)] \(((\widetilde{l_{j}},0),I_{n+1}) \in \Gamma\);
	\item[(ii)]\(((\widetilde{l_{j}},y_{1}),\sigma)\in \Gamma\), for some \(y_{1}\in\mathbb{R}\); 
	\item[(iii)]\((0,y_{0})\), and \((\widetilde{l_{j}},\frac{p}{q}y_{0}) \in \mathcal{L}\), for some \(p, \ q\) nonzero integers. 
\end{itemize}
Condition \( (iii)\)  is a stronger version of condition III in  remark~\ref{remProjLattice}. The other conditions follow from  remark~\ref{remProjLattice}.
Any generator \(\widetilde{l_{{j}}}\) of  \(\widetilde{\mathcal{L}} \)  such that 
 \( (\widetilde{l_{{j}}}, y) \in \mathcal{L}\) must satisfy either \( (i)\) or \( (iii) \) above,
 because \((0,y_{0})\)  is rational with respect to \( \mathcal{L}\).
 Any other generator of  \(\widetilde{\mathcal{L}} \)  must satisfy  \( (ii)\), proving our claim.
 
 Conversely, if one of the conditions \( (i) \) or \( (ii) \) is true,  then  \( (2v,0)\in\mathcal{L}\), using remark~\ref{remProjLattice}.
 If condition \( (iii)\) holds for some \(j \in \lbrace1,\cdots, n\rbrace\), then \((0,y_{0}), \ (\widetilde{l_{j}},\frac{p_{j}}{q_{j}}y_{0}) \in \mathcal{L}\), where \(\ p_{j}, \ q_{j}\) are nonzero integers. Since  \(\mathcal{L}\) \ is a lattice \((q_{j}\widetilde{l_{j}},0) \in \mathcal{L}\).
 Therefore   \(\widetilde{\mathcal{L}}_s\) is rationally compatible with \(\mathcal{L}\).
\end{proof}
As an illustration, take  \(\Gamma=\mathcal{L}=\langle (0,1), (\sqrt{2},1/2) \rangle_\mathbb{Z}\), for which 
 \(\widetilde{\mathcal{L}}_s=\langle(\sqrt{2},0)\rangle_\mathbb{Z}\)  is always rationally compatible with \(\mathcal{L}\), independently of $y_0$.
 Another example is given by \(\Gamma=\mathcal{L}=\langle (0,1), (1,\sqrt{2}) \rangle_\mathbb{Z}\). For \(y_0=1\) we have that  \(\widetilde{\mathcal{L}}_s=\langle(1,0)\rangle_\mathbb{Z}\) is  not rationally compatible with \(\mathcal{L}\),
 whereas  for \(y_0\notin\mathbb{Z}\), we get
  \(\widetilde{\mathcal{L}}_s=\{(0,0)\}\).

	For 3-dimensional lattices, 	
and if the generators of \(\widetilde{\mathcal{L}}\) are related by an orthogonal transformation, then we can 
remove the condition on \( (0,y_0)\) from the statement of proposition~\ref{propRational}, at the price of having some more complicated conditions.
This provides an alternative means of obtaining rational compatibility.
%
%
	
Our starting point is a specific 2-dimensional lattice \tilL \ and
 we want to characterize the 3-dimensional lattices 
 \( \mathcal{L}\) 
 that project onto this. The first step is to establish that  
 \( \mathcal{L}\) 
 must have a non-trivial intersection with the plane \(X0Y = \lbrace(x,y,0); \ x, \ y \in \mathbb{R} \rbrace\).
		
\begin{theorem}\label{thRational2}
Let \(\Gamma\) be a crystallographic group with a lattice \(\mathcal{L}\subset \mathbb{R}^{3}\)  such that its projection 
 \( \widetilde{\Gamma}=\Pi_{y_{0}}({\Gamma})\)  has
 a plane lattice  \tilL \ \(=\Pi_{y_{0}}(\mathcal{L})\) generated by two linearly independent vectors 
\begin{center}
\(\widetilde{l_{1}}\) and \(\widetilde{l_{2}} = \rho \widetilde{l_{1}}\)
\end{center}
for \(\rho\) in the point group, \(\widetilde{J} \), of  \
 \(\widetilde{\Gamma}\).

Then the suspension
 \(\widetilde{\mathcal{L}}_s\subset \mathbb{R}^{3}\)
is rationally compatible with \(\mathcal{L}\) if
%
for each \(v \in \lbrace\widetilde{l_{1}}, \ \widetilde{l_{2}}\rbrace\) one of the following conditions holds:
\begin{itemize}
	\item[a.] \(((v,0),I_{3}) \in \Gamma\);
	\item[b.] \(((v,y_{1}),\sigma)\in \Gamma\), for some \(y_{1}\in\mathbb{R}\); 
	\item[c.] \((v,y_{1})\in \mathcal{L}\), for some \(y_{1}\in\mathbb{R}\). 
	\end{itemize}
	
\end{theorem}

	That the conditions are also necessary is immediate from remark~\ref{remProjLattice}.
	Note that the statement of theorem~\ref{thRational2} excludes oblique and primitive rectangular lattices.

\begin{proof}
%
	Since \(\widetilde{l_{2}} = \rho \widetilde{l_{1}}\), it is sufficient to show that one of \(r(\widetilde{l_{1}},0)\) or \(r(\widetilde{l_{2}},0)\) is in \( \mathcal{L}\) . To see this, suppose, without loss of generality, that \(r(\widetilde{l_{1}},0) \in \mathcal{L}\). Then since \(\rho \in \widetilde{J}\), by corollary~\ref{corolHolohedry1}, either \(\rho_{+} \in H_{\mathcal{L}}\) or \(\rho_{-} \in H_{\mathcal{L}}\). As \(\rho_{+}(r\widetilde{l_{1}},0) = \rho_{-}(r\widetilde{l_{1}},0) = (r\widetilde{l_{2}},0)\), it implies that \(r(\widetilde{l_{2}},0) \in \mathcal{L}\) and therefore, $\mathcal{L}$ has a sublattice \(\mathcal{L}_{r}\).
	
	If for some \(v \in \lbrace\widetilde{l_{1}}, \ \widetilde{l_{2}}\rbrace\) one of the conditions \(a\) or \(b\) is true then, by remark~\ref{remProjLattice}, \((rv,0)\in\mathcal{L}, \ \mbox{for} \ r = 1 \ \mbox{or} \ r = 2\). Hence, all that remains to prove is the case when  $\widetilde{l_{1}}$ and $\widetilde{l_{2}}$ only satisfy condition \(c\).
	
	By hypothesis,
\begin{equation}\label{hypothesis}
 (\widetilde{l_{1}},y_{1}) \ \mbox{and} \ (\widetilde{l_{2}},y_{2}) \ \mbox{are in} \ \mathcal{L}, \ \mbox{for some} \ y_{1}, \ y_{2} \ \in \mathbb{R} 
\end{equation}
this implies that 
\begin{equation}\label{conclusion1}
(\widetilde{l_{1}} + \widetilde{l_{2}}, y_{1} + y_{2}) \in \mathcal{L}
\end{equation}

	Using \eqref{hypothesis} and corollary~\ref{corolHolohedry1}
\[\mbox{either} \ \rho_{+}(\widetilde{l_{1}},y_{1}) = (\widetilde{l_{2}},y_{1}) \in \mathcal{L} \ \mbox{or} \ \rho_{-}(\widetilde{l_{1}},y_{1}) = (\widetilde{l_{2}},-y_{1}) \in \mathcal{L}.\]

	If \((\widetilde{l_{2}},y_{1}) \in \mathcal{L}\) then
\[(\widetilde{l_{2}},y_{1}) + (\widetilde{l_{2}},y_{2}) = (2\widetilde{l_{2}},y_{1} + y_{2})\in \mathcal{L}\]
thus, using \eqref{conclusion1}
\[(\widetilde{l_{1}} + \widetilde{l_{2}}, y_{1} + y_{2}) - (2\widetilde{l_{2}},y_{1} + y_{2}) = (\widetilde{l_{1}} - \widetilde{l_{2}},0) \in \mathcal{L}\]	

	Since \(\lbrace\widetilde{l_{1}}, \ \widetilde{l_{2}}\rbrace\) is a basis to \(\widetilde{\mathcal{L}}\) and \(\rho \in \widetilde{J}\) then
\[\rho(\widetilde{l_{1}} - \widetilde{l_{2}}) = m\widetilde{l_{1}} + n\widetilde{l_{2}}, \ m, \ n \in \mathbb{Z} \]
where \(m, \ n\) are not both equal to zero. Suppose that \(n \neq 0\), then
\begin{equation}\label{conclusion2}
 n(\widetilde{l_{1}} - \widetilde{l_{2}},0),\ (m\widetilde{l_{1}} + n\widetilde{l_{2}}, 0) \in \mathcal{L}
\end{equation} 	
implying that the sum of these last two vectors \(((n+m)\widetilde{l_{1}},0) \in \mathcal{L}\). Therefore, if \(n \neq -m\), \(\mathcal{L}_{r}\) is a sublattice of \( \mathcal{L}\) , where \(r = m + n \in \mathbb{Z}\). If \(n = -m\), we subtract the two expressions in \eqref{conclusion2} to get \((2n\widetilde{l_{1}},0)\in \mathcal{L}\).

	If \((\widetilde{l_{2}},-y_{1}) \in \mathcal{L}\) then 
\[-(\widetilde{l_{2}},-y_{1}) + (\widetilde{l_{2}},y_{2}) = (0, y_{1} + y_{2})\in \mathcal{L}\]
thus, using \eqref{conclusion1}
\[(\widetilde{l_{1}} + \widetilde{l_{2}}, y_{1} + y_{2}) - (0, y_{1} + y_{2}) = (\widetilde{l_{1}} + \widetilde{l_{2}},0) \in \mathcal{L}\]	

	An analogous argument applied to \(\rho(\widetilde{l_{1}} + \widetilde{l_{2}},0)\) finishes the proof.
\end{proof}	

	Let \(\mathcal{L}\subset\mathbb{R}^{3}\) be a lattice and \(P\subset\mathbb{R}^{3}\) be a plane such that \(
P\cap\mathcal{L}\neq\varnothing\). Given $v\in P\cap\mathcal{L}$ there is a rotation $\gamma \in O(3)$ such that $\gamma (P - v)$ is the plane $X0Y = \lbrace(x,y,0); \ x, \ y \in \mathbb{R} \rbrace$. Then we define the \emph{\(y_{0}\)-projection} of \( \mathcal{L}\) \ into $P$ as the lattice $\gamma^{-1}(\widetilde{\mathcal{L}})\subset E(2)$ where $\widetilde{\mathcal{L}}$ is the symmetry group of \(\Pi_{y_{0}}(\mathcal{X}_{\gamma(\mathcal{L}-v)})\).	

	We say that \emph{the \(y_{0}\)-projection} of  \( \mathcal{L}\) \ into the plane $P$ \emph{is a hexagonal plane lattice} if and only if the lattice \(\widetilde{\mathcal{L}}\)	admits as its holohedry a group isomorphic to \(D_{6}\).
		
	Our main result is the following theorem.
Note that the hypothesis of having a 3-fold rotation is not restrictive when one is looking for projections yielding a pattern with  \(D_{6}\) symmetry.
	
\begin{theorem}\label{thHexagonal}
	Let $\mathcal{L}\subset \mathbb{R}^{3}$ be a lattice of a crystallographic group $\Gamma$. Suppose for some $y_{*}$ the   group  \(\Pi_{y_{*}}(\Gamma)\) contains a 3-fold rotation. Then for any $y_{0}$, the $y_{0}$-projection of  \( \mathcal{L}\) \ into the plane $P$ is a hexagonal plane lattice if and only if:
\begin{itemize}
\item[(1)] \(P\cap \mathcal{L}\) contains at least two elements;
\item[(2)] there exists \(\beta \in H_{\mathcal{L}}\) such that: 
\begin{itemize}
\item $\beta$ is a 3-fold rotation; 
\item $P$ is $\beta$-invariant.
\end{itemize}	
\end{itemize}
\end{theorem}	
	
\begin{proof}
	Suppose first that \((0,0,0)\in P \cap \mathcal{L}\). 
	To show that  conditions (1) and (2) are necessary let us consider, without loss of generality, that  \(P = X0Y\). Therefore, the conditions hold by theorem~\ref{thRational2}.


To prove that  condition (1) and (2) are sufficient consider the 3-fold rotation $\beta \in H_{\mathcal{L}}$. By \cite{miller72}, theorem 2.1 and the proof of the crystallographic restriction theorem, in the same reference, there exists only one subspace of dimension 2 invariant by $\beta$. Such a plane is the plane perpendicular to its rotation axis.  So, let $P$ be this plane.

	Since \(P\cap \mathcal{L}\neq \lbrace (0,0,0) \rbrace\), let v be a nonzero element of minimum length in \(P\cap \mathcal{L}\) and consider the lattice $\mathcal{L}^{'} = \langle v, \beta v\rangle_{\mathbb{Z}}$. As $\beta$ has order 3 the sublatice \(\mathcal{L}^{'}\) is a hexagonal plane lattice.
	
	To finish the proof, consider $y_{0} \in \mathbb{R}$; we  prove that the $y_{0}$-projection of \(\mathcal{L}\) into the plane $P$ is a hexagonal plane lattice. Let $(0, \alpha)  \in \Pi_{y_{*}}(\Gamma)$, where $\alpha$ is a 3-fold rotation. Then by theorem~\ref{thPinho}, one of the conditions holds:
\begin{itemize}
\item [1] $(0, \alpha_{+}) \in \Gamma$;
\item [2] $((0,y_{*}), \alpha_{-})\in \Gamma$;
\item [3] $(0,y_{*}) \in \mathcal{L}$ and either $((0,y_{1}), \alpha_{+})$ or $((0,y_{1}), \alpha_{-})$ is in $\Gamma$.
\end{itemize}	

	Since the order of $\alpha$ is finite , we have  that either $(0, \alpha_{+}) \in \Gamma$ or $(0, \alpha_{-}) \in \Gamma$. Then, the result follows  by condition 2 of corollary~\ref{corolHolohedry2}.
	
	
	If \((0,0,0) \notin P \cap \mathcal{L}\), note that the proof can be reduced to the previous case by a translation.
%
\end{proof}

\begin{rem}\label{remHexagonal}
 \emph{Theorem~\ref{thHexagonal} shows that a possible way to obtain patterns with hexagonal symmetry, by $y_{0}$-projection, is to project the functions $f\in \mathcal{X}_{\mathcal{L}}$ in a plane invariant by the action of some element $\beta \in H_{\mathcal{L}}$ with order 
  three.
 After finding one of those planes, in order to obtain projections as in definition~\ref{defProjection}, we only need to change coordinates. The reader can see an example with the bcc lattice in \cite{gomes99}.}

	\emph{We are grateful to an anonymous referee who pointed out that for certain specific widths of the projection this can be obtained by other means. However, in these cases the symmetry group of the projected functions, for 
most
$y_{0}$, has a very small point group and this is not interesting for the study of bifurcating patterns. 
More specifically, we are interested in relating hexagonal patterns of different complexity in solutions of the same differential equation with symmetry.
As the projection width $y_0$ varies, one may obtain hexagonal patterns with different symmetry groups, corresponding to different patterns, as illustrated in the figures at the end of this article.
Bifurcation occurs via symmetry-breaking and hence, more symmetry (a bigger point group) makes the bifurcation problem more interesting.
}

\end{rem}
	
	As a consequence of theorem~\ref{thHexagonal}, we are able to list all the Bravais lattices that may be projected to produce a 2-dimensional hexagonal pattern.
	
\begin{theorem}\label{thBravais}
The Bravais lattices that project to a hexagonal plane lattice, under the conditions of theorem~\ref{thHexagonal} are:
\begin{itemize}
\item[1.] Primitive Cubic lattice;
\item[2.] Body-centred Cubic lattice;
\item[3.] Face-centred Cubic lattice;
\item[4.] Hexagonal lattice; and
\item[5.] Rhombohedral lattice. 
\end{itemize}
Moreover, up to change of coordinates, for the first three lattices the plane of projection must be parallel to one of the planes in Table 1. For the Hexagonal and Rhombohedral lattices the plane of projection must be parallel to the plane \(X0Y\).
\end{theorem}

\begin{proof}
	It is immediate from theorem~\ref{thHexagonal} that we can exclude the following Bravais lattices: triclinic, monoclinic, orthorhombic and tetragonal, since the holohedries of these lattices do not have elements of order 
three.
	
	To see if the other Bravais lattices have hexagonal projected symmetries, we need to examine the rotations of order three and six in their holohedries and see if the plane perpendicular to their rotation axes intersects the lattice.

	The group of rotational symmetries of the Cubic lattice (as well as the Body Centred Cubic lattice and the Face Centred Cubic lattice) is isomorphic to $S_{4}$, the group of permutation of four elements. So, in the holohedry of the Cubic lattice we only have rotations of order one, two or three. 
	Consider a systems of generators for a representative for the Cubic lattice  \( \mathcal{L}\) , in the standard basis of \(\mathbb{R}^{3}\), given by:
\[(1,0,0), \ (0,1,0), \ (0,0,1)\]	
	
Then, the matrix representation of the rotations of order 3 in \(H_{\mathcal{L}}\) are: 
\[\gamma_{1} = \left(\begin{array}{ccc}
 0  & -1 & 0 \\ 
 0  &  0 & 1 \\ 
-1  &  0 & 0
\end{array}\right), \ \ \gamma_{2} = \left(\begin{array}{ccc}
 0 & 1 &  0 \\ 
 0 & 0 & -1 \\ 
-1 & 0 &  0
\end{array}\right),\]
\[\gamma_{3} = \left(\begin{array}{ccc}
 0 &  0 & 1 \\ 
-1 &  0 & 0 \\ 
 0 & -1 & 0
\end{array}\right),\ \ \gamma_{4} = \left(\begin{array}{ccc}
0 & 1 & 0 \\ 
0 & 0 & 1 \\ 
1 & 0 & 0
\end{array}\right)\]

	Two-dimensional spaces perpendicular to the rotation axis of each one of these rotations are given in Table 1.
	
\begin{table}
\caption{Two-dimensional spaces perpendicular to the rotation axis of each one of the rotations \(\gamma_{i}\). Here we denote by \(\langle v \rangle, v \in \mathbb{R}^{3}\) the subspace generated by \(v\).}
    \begin{tabular}{ccl}
   \textbf{Rotation} & \textbf{Rotation Axis} & \textbf{Perpendicular Plane}  \\ 
    \(\gamma_{1}\) & \(\langle (1,1,-1) \rangle\)& \(P_{1}=\lbrace (x,y,z); z = x +y\rbrace\)  \\ 
    \(\gamma_{2}\) & \(\langle (1,-1,-1) \rangle\) & \(P_{2}=\lbrace (x,y,z); z = x - y\rbrace\)  \\ 
    \(\gamma_{3}\) & \(\langle (1,-1,1) \rangle\) & \(P_{3} = \lbrace (x,y,z); z = -x + y\rbrace\)  \\ 
    \(\gamma_{4}\) & \(\langle (1,1,1) \rangle\)  & \(P_{4} = \lbrace (x,y,z); z = -(x + y)\rbrace\)
    \end{tabular}
\end{table}

	This means that for the fist three lattices in the list, the projection of functions \(f \in \mathcal{X}_{\mathcal{L}}\) into a plane have hexagonal symmetries only if the plane is parallel to one of the plane subspaces given in table 1. 
	
	Consider now a 3-dimensional Hexagonal lattice. Its group of rotational symmetries has order twelve and it has a subgroup of order six consisting of the rotational symmetries of the Rhombohedral lattice. 
	
	Let the representatives for the Hexagonal and Rhombohedral lattices, be generated by:
\[(1,0,0), (\frac{1}{2},\frac{\sqrt{3}}{2},0), (0,0,c) \ \ c \neq 0,\ \pm 1. \]

\[(1,0,1), (-\frac{1}{2},\frac{\sqrt{3}}{2},1), (-\frac{1}{2},-\frac{\sqrt{3}}{2},1)\]
respectively. Then, the twelve rotations in the holohedry of the Hexagonal lattice are generated by:
\[\rho_{z} = \left(\begin{array}{ccc}
	\frac{1}{2} & -\frac{\sqrt{3}}{2} & 0 \\ 
	\frac{\sqrt{3}}{2} & \frac{1}{2}  & 0 \\ 
	0 & 0 & 1
	\end{array}\right) \ \mbox{and} \ \gamma_{x} = \left(\begin{array}{ccc}
	1 & 0 & 0 \\ 
	0 & -1 & 0 \\ 
	0 & 0 & -1
	\end{array}\right)\]
	
	The generators of the group of rotational symmetries of the Rhombohedral lattice are then \(\rho_{z}^{2}\) and \(\gamma_{x}\).
	
	We conclude that the only rotations of order 6 in the holohedry of the Hexagonal lattice are \(\rho_{z}\) and \(\rho_{z}^{5}\), and of order 3 \(\rho_{z}^{2}\) and \(\rho_{z}^{4}\). 
	
	Therefore, the \(y_{0}\)-projection of the Hexagonal and Rhombohedral lattices is a hexagonal plane sublattice if and only if the \(y_{0}\)-projection is made into a plane parallel to the plane \(X0Y\). 
\end{proof}

\section{Hexagonal projected symmetries of the Primitive Cubic lattice}

	We conclude the article with an example to illustrate the hexagonal symmetries obtained by \(z_{0}\)-projection of functions with periods in the Primitive Cubic lattice, for all \(z_{0} \in \mathbb{R}\).
	
	Consider a 3-dimensional crystallographic group, \(\Gamma = \mathcal{L}\dotplus H_{\mathcal{L}}\), where  \( \mathcal{L}\) \ is the Primitive Cubic Lattice generated by the vectors \((1,0,0), (0,1,0) \ \mbox{and} \ (0,0,1)\) over \(\mathbb{Z}\), and \(H_{\mathcal{L}}\) its holohedry. 
		
	 Without loss of generality, consider the projection of \(\Gamma\) on \(P_{1}\) (see Table 1). 
	
	From theorem~\ref{thHexagonal}, the Cubic lattice has a hexagonal plane sublattice that intersects \(P_{1}\). This sublattice is generated by:
\begin{equation} \label{sublattice}
(0,1,1), \ (1,0,1)
\end{equation}	

	To make our calculations easier and to set up the hexagonal symmetries in the standard way consider the new basis \(\lbrace (0,1,1), (1,0,1), (0,0,1)\rbrace\) for the lattice  \( \mathcal{L}\) . Now multiply  \( \mathcal{L}\) \ by the scalar \(\frac{1}{\sqrt{2}}\) in order to normalise the vectors of \eqref{sublattice}. With these changes the crystallographic group \(\Gamma\) has the new translational subgroup generated by the vectors:
\[
v_{1} = (0,\frac{1}{\sqrt{2}},\frac{1}{\sqrt{2}}), \ \ v_{2} = (\frac{1}{\sqrt{2}},0,\frac{1}{\sqrt{2}}),\ \ v_{3} = (0,0,\frac{1}{\sqrt{2}})
\]

	Projection of \(\Gamma\) on \(P_{1}\), as in definition~\ref{defProjection}, can be done after a change of coordinates that transforms \(P_{1}\) into \(X0Y\). Consider that change given by the orthonormal matrix 
\[A = \left(\begin{array}{ccc}
         0         & \frac{1}{\sqrt{2}} &  \frac{1}{\sqrt{2}} \\ 
\frac{2}{\sqrt{6}} & \frac{-1}{\sqrt{6}} & \frac{1}{\sqrt{6}} \\ 
\frac{1}{\sqrt{3}} & \frac{1}{\sqrt{3}} & \frac{-1}{\sqrt{3}}
\end{array}\right)\] 
	
	Then, in the new system of coordinates \(X = Ax\), we obtain the base for the Primitive Cubic lattice given by:
\begin{equation} \label{lattice}
 l_{1}=(1,0,0), l_{2}=(\frac{1}{2},\frac{\sqrt{3}}{2},0), l_{3}=(\frac{1}{2},\frac{\sqrt{3}}{6},\frac{-\sqrt{6}}{6})
\end{equation}
Observe that we changed the position of  \( \mathcal{L}\) \ as prescribed by theorem~\ref{thHexagonal}.

	We proceed to describe the symmetries of the space \(\Pi_{z_{0}}(\mathcal{X}_{\Gamma})\), for each \(z_{0} \in \mathbb{R}\). For this, we need to obtain the subgroups \(\widehat{\Gamma}\) and \(\Gamma_{z_{0}}\) of \(\Gamma\). Denote by \(\Sigma_{z_{0}} = \mathcal{L}_{z_{0}}\dotplus J_{z_{0}}\) the subgroup of \(E(2)\) of all symmetries of \(\Pi_{z_{0}}(\mathcal{X}_{\Gamma})\).
	
	It is straightforward to see that the elements of \(\Gamma\) with orthogonal part \(\alpha_{\pm}\) are in the group
\[\widehat{\Gamma} = \lbrace((v,z),\rho);(v,z)\in \mathcal{L}, \ \ \rho \in \widehat{J}\rbrace\]
where \(\widehat{J}\) is the group generated by
\[\gamma = \left(\begin{array}{ccc}
\frac{1}{2}        &  -\frac{\sqrt{3}}{2} &  0 \\ 
\frac{\sqrt{3}}{2} &  \frac{1}{2}        &  0 \\ 
    0              &      0              & -1
\end{array}\right) \ \ \mbox{and} \ \ \kappa = \left(\begin{array}{ccc} 
-1 & 0 & 0 \\ 
0  & 1 & 0 \\ 
0  & 0 & 1
\end{array}\right) \] 
and the group \(\Gamma_{z_{0}}\) has a subgroup \(H = \overline{\mathcal{L}} \ \dotplus \ \overline{J}\), for all \(z_{0} \in \mathbb{R}\), where \(\overline{\mathcal{L}}\) is the translation subgroup \(\overline{\mathcal{L}} = \langle l_{1}, \ l_{2}\rangle_{\mathbb{Z}}\) and \(\overline{J}\) is the subgroup generated by \(((0,0,0),\kappa)\) and \(((0,0,0),-\gamma)\). Using statement I of theorem~\ref{thPinho}, for all \(z_{0} \in \mathbb{R}\) all the functions \(f \in \Pi_{z_{0}}(\mathcal{X}_{\Gamma})\) are \((1,0)\), and \((\frac{1}{2},\frac{\sqrt{3}}{2})\) periodic and invariant for the action of 
\[\kappa^\prime=\left(\begin{array}{rr} 
-1 & 0  \\ 
0  & 1  
\end{array}\right) \ \mbox{and}
 \ -\gamma^\prime=\left(\begin{array}{cc}
-\frac{1}{2}        &  \frac{\sqrt{3}}{2}  \\ 
-\frac{\sqrt{3}}{2} &  -\frac{1}{2}         
\end{array}\right)\]

	In Table 2 we list the group \(\Gamma_{z_{0}}\), for each \(z_{0} \in \mathbb{R}\), and describe the respective projected symmetries.
	 
\begin{table}
\caption{Projection of \(\Gamma = \mathcal{L}\dotplus H_{\mathcal{L}}\), for each \(z_{0} \in \mathbb{R}\).}
    \begin{tabular}{l|l|l}
   \textbf{\(z_{0} \in \mathbb{R}\)} & \textbf{\(\Gamma_{z_{0}}\)} & \textbf{\(\Sigma_{z_{0}}=\mathcal{L}_{z_{0}}\dotplus J_{z_{0}}\)}  \\ 
\hline   
    \(z_{0} = \frac{3n}{\sqrt{6}}, \ n \in \mathbb{Z}\backslash\lbrace0\rbrace\) \\ then \((0,0,\frac{3n}{\sqrt{6}})\in\mathcal{L}\) & \(\Gamma_{z_{0}} = \widehat{\Gamma}\) &  \shortstack[l]{\(\mathcal{L}_{z_{0}} = \langle(\frac{1}{2},\frac{\sqrt{3}}{6}),(\frac{1}{2},\frac{-\sqrt{3}}{6})\rangle_{\mathbb{Z}}\)\\ \(J_{z_{0}} = D_{6} = \langle\gamma^{'},\ \kappa^{'}\rangle\)} \\
 \hline    
    \(z_{0} = \frac{3n-1}{\sqrt{6}}, \ n \in \mathbb{Z}\backslash\lbrace0\rbrace\) \\ then \((\frac{1}{2},\frac{\sqrt{3}}{6},\frac{3n-1}{\sqrt{6}})\in\mathcal{L}\) & \shortstack[l]{ \(\Gamma_{z_{0}}\) contains \\ \(((\frac{1}{2},\frac{\sqrt{3}}{6},\frac{3n-1}{\sqrt{6}}), \gamma)\)} & \shortstack[l]{\(\mathcal{L}_{z_{0}} = \langle(1,0),(\frac{1}{2},\frac{\sqrt{3}}{2})\rangle_{\mathbb{Z}}\)\\ \(J_{z_{0}} = \langle((\frac{1}{2},\frac{\sqrt{3}}{6}),\gamma^{'}),\ \kappa^{'}\rangle\)}  \\ 
\hline    
     \(z_{0} = \frac{3n+1}{\sqrt{6}}, \ n \in \mathbb{Z}\backslash\lbrace0\rbrace\) \\ then \((\frac{1}{2},\frac{-\sqrt{3}}{6},\frac{3n+1}{\sqrt{6}})\in\mathcal{L}\)  &  \shortstack[l]{ \(\Gamma_{z_{0}}\) contains\\ \(((\frac{1}{2},\frac{-\sqrt{3}}{6},\frac{3n+1}{\sqrt{6}}),\gamma)\)} & \shortstack[l]{\(\mathcal{L}_{z_{0}} = \langle(1,0),(\frac{1}{2},\frac{\sqrt{3}}{2})\rangle_{\mathbb{Z}}\)\\ \(J_{z_{0}} = \langle((\frac{1}{2},\frac{-\sqrt{3}}{6}),\gamma^{'}),\ \kappa^{'}\rangle\)}   \\ 
\hline    
    For \(z_{0}\) different \\ of the cases before  & \(\Gamma_{z_{0}} = H\)  & \shortstack[l]{\(\mathcal{L}_{z_{0}} = \langle(1,0),(\frac{1}{2},\frac{\sqrt{3}}{2})\rangle_{\mathbb{Z}}\)\\ \(J_{z_{0}} = \langle -\gamma^{'},\ \kappa^{'}\rangle\)}
    \end{tabular}
\end{table}
	 
	 We assume that all the functions \(f:\mathbb{R}^{3} \rightarrow \mathbb{R}\) in \(\mathcal{X}_{\mathcal{L}}\) admit a unique formal Fourier expansion in terms of the waves
\[w_{k}(x,y,z) = \exp(2\pi i \langle k, (x,y,z)\rangle)\]
where \(k\) is a \emph{wave vector} in the dual lattice, \(\mathcal{L}^{*} = \lbrace k \in \mathbb{R}^{3}; \ \langle k, l_{i} \rangle \in \mathbb{Z}, \ i = 1, \ 2, \ 3 \rbrace\), of \( \mathcal{L}\) \ given in \eqref{lattice}, with \emph{wave number} \(\vert k\vert\),  where \((x,y,z) \in \mathbb{R}^{3}\) and \(\langle\cdot,\cdot\rangle\) is the usual inner product in \(\mathbb{R}^{3}\). Thus,  
\[f(x,y,z) = \displaystyle \sum_{k\in \mathcal{L}^{*}} z_{k}w_{k}(x,y,z)\]
where \(z_{k}\) is the Fourier coefficient, for each \(k\in \mathcal{L}^{*}\), and with the restriction \(z_{-k} = \overline{z_{k}}\).

	Therefore, we can write
\[\mathcal{X}_{\mathcal{L}} = \displaystyle \bigoplus_{k \in \mathcal{L}^{'}} V_{k}\]
for
\[\mathcal{L}^{'} = \lbrace k = (k_{1},k_{2})\in \mathcal{L}^{*};\ k_{1}>0 \ \mbox{or} \  k_{1} = 0 \ \mbox{and} \ k_{2}>0 \rbrace \] 
and 
\[ V_{k} = \lbrace Re(zw_{k}(x,y,z)); \ z \in \mathbb{C}\rbrace \cong \mathbb{C}\]

	Note that \(\mathcal{X}_{\Gamma}\) is a subspace of \(\mathcal{X}_{\mathcal{L}}\).

	We say that the space 
\[\bigoplus_{\mid K\mid = a} V_K= V_{K_{1}}\oplus V_{K_{2}}\oplus\cdots\oplus V_{K_{s}}\]
 is a $2s$-dimensional representation of the action of $\Gamma$ on the space \(\mathcal{X}_{\mathcal{L}}\).
	
	A straightforward calculation shows that the function
\begin{equation} \label{function}
u(x,y,z) = \displaystyle \sum_{\mid k \mid = \sqrt{2}} \exp (2\pi i k\cdot (x,y,z))
\end{equation}
is \(\Gamma\)-invariant.

	The contour plots of the projections of $u$ are shown in  Figure~\ref{figPC}, with the symmetries given in Table 2. In \cite{dionne93} it is shown that the function $u$ belongs to a 6-dimensional representation.

\begin{figure}	\label{figPC}
(a){\includegraphics[width=0.20\textwidth]{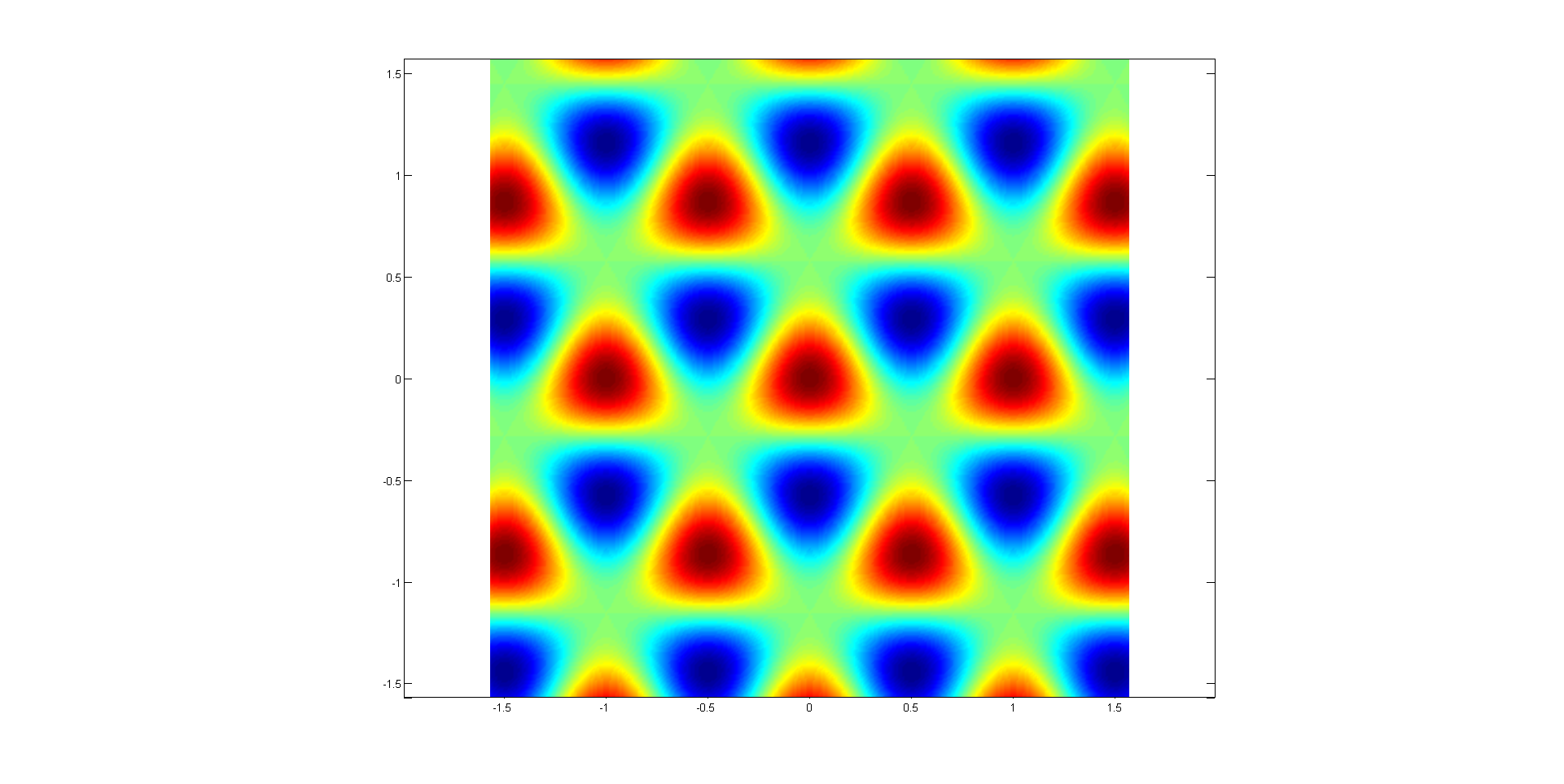}}
(b){\includegraphics[width=0.2\textwidth]{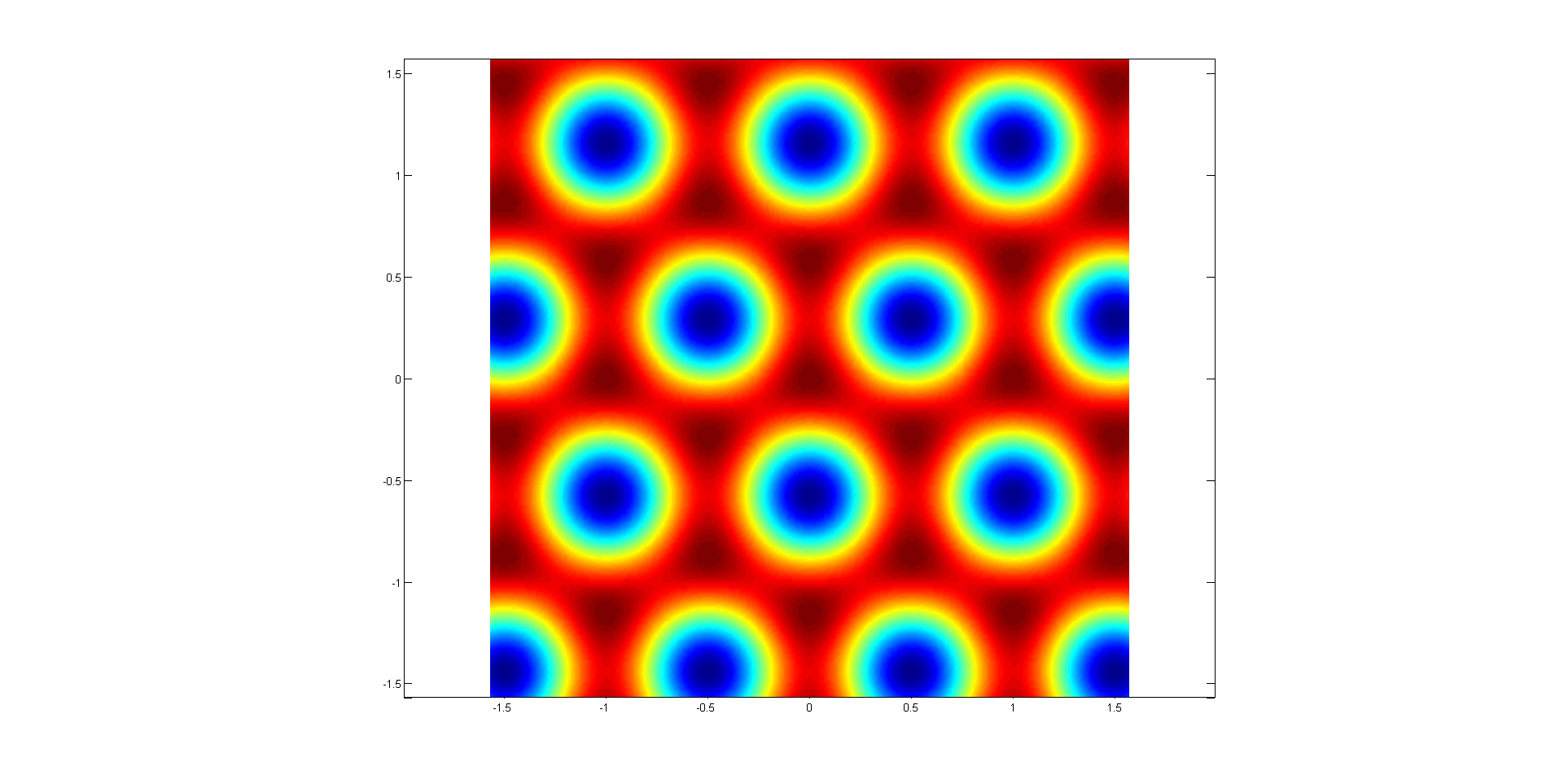}}
(c){\includegraphics[width=0.2\textwidth]{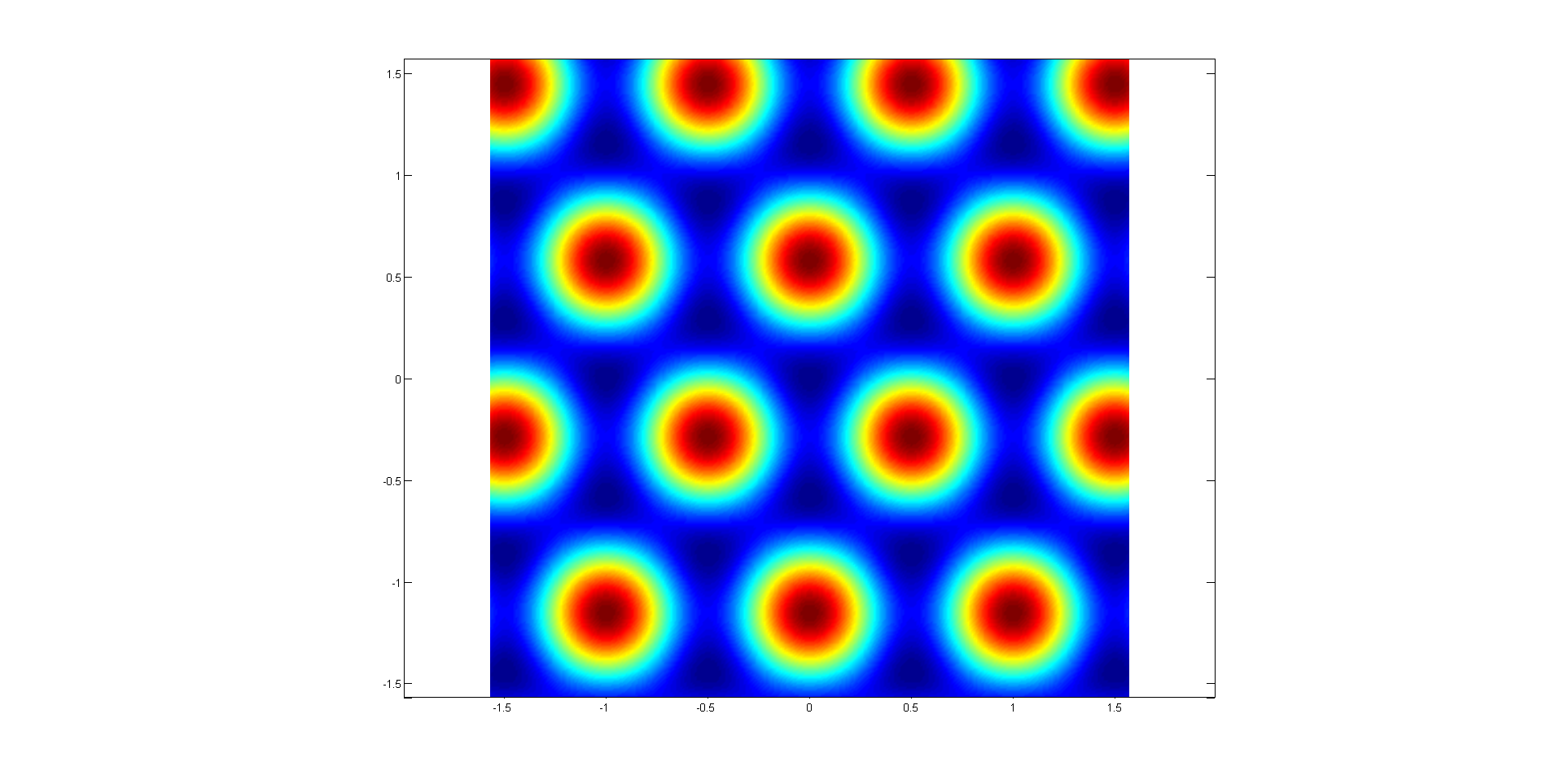}}
(d){\includegraphics[width=0.2\textwidth]{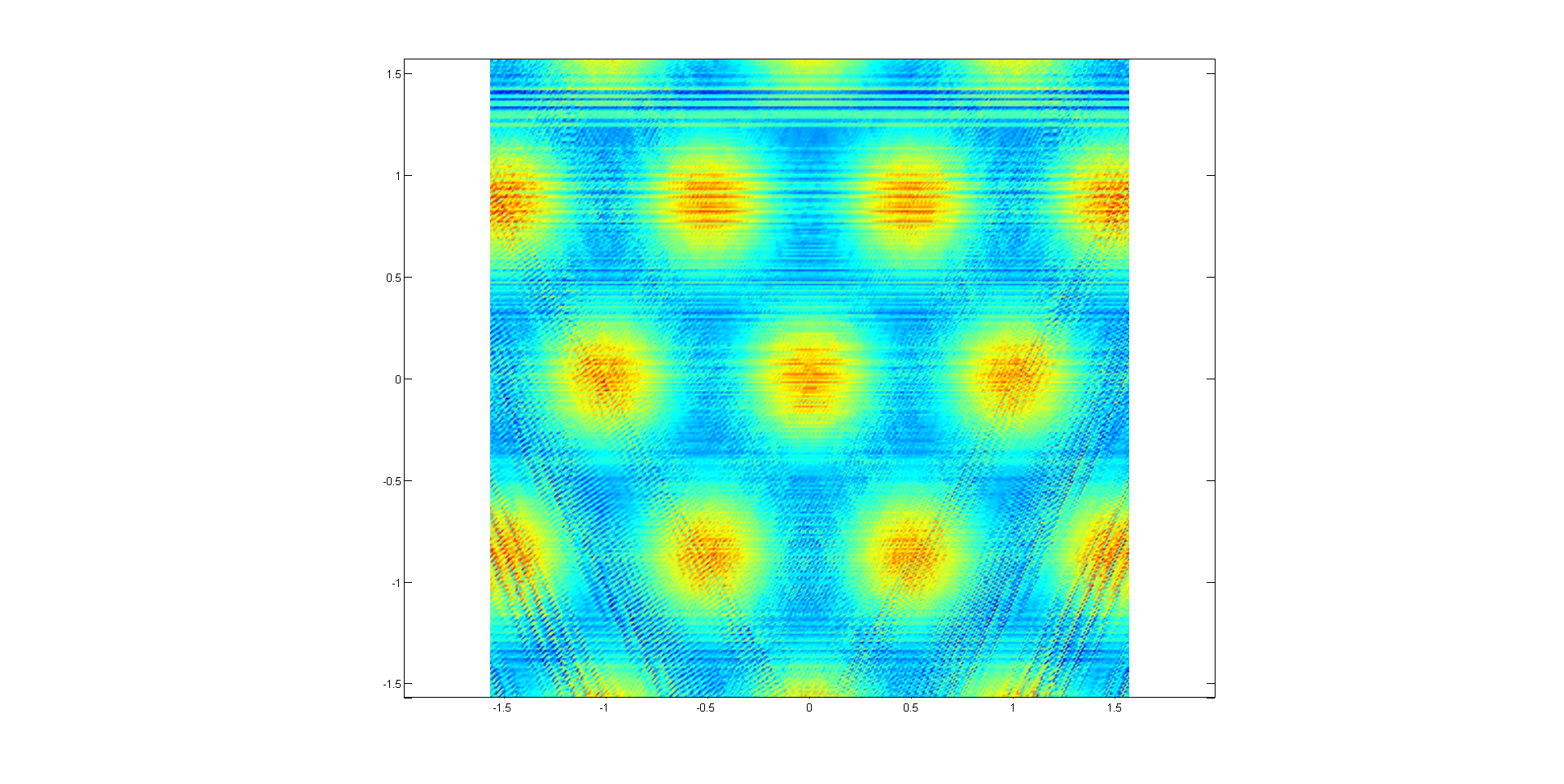}}
\caption{ Projection of pattern $u$ in a 6-dimensional representation with Primitive Cubic lattice periodicity.  
Contour plots of the integral of $u$ over different depths \(z_{0}\). (a) \(z_{0} = \frac{1}{ 2\sqrt{6}}\). (b) \(z_{0} = \frac{1}{\sqrt{6}}\). (c) \(z_{0} = \frac{2}{\sqrt{6}}\). (d) \(z_{0} = \frac{3}{ \sqrt{6}}\).  The same pictures occur for the projection of a strip of half this height of a pattern in a 8-dimensional representation with face centred cubic periodicity.} 
\end{figure} 

	The Body Centred Cubic lattice shows a different configuration, illustrated in Figure~\ref{figBCC}.
	
\begin{figure}	\label{figBCC}
(a){\includegraphics[width=0.2\textwidth]{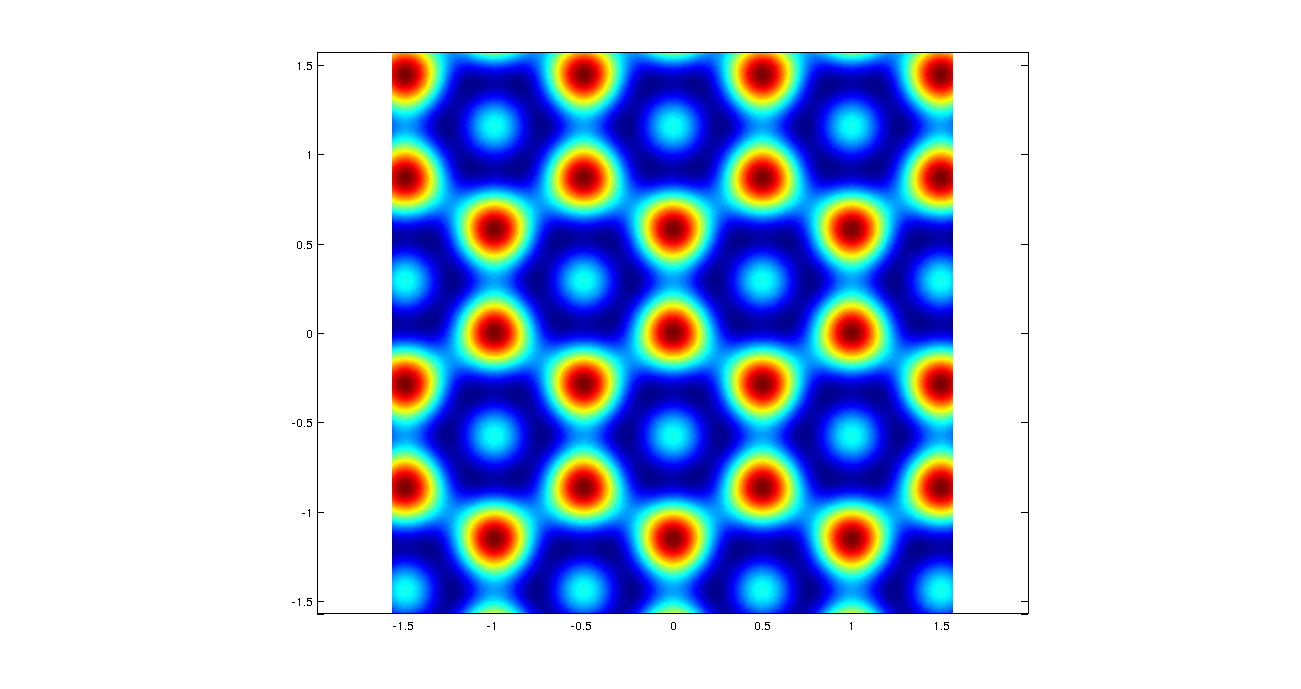}}
(b){\includegraphics[width=0.2\textwidth]{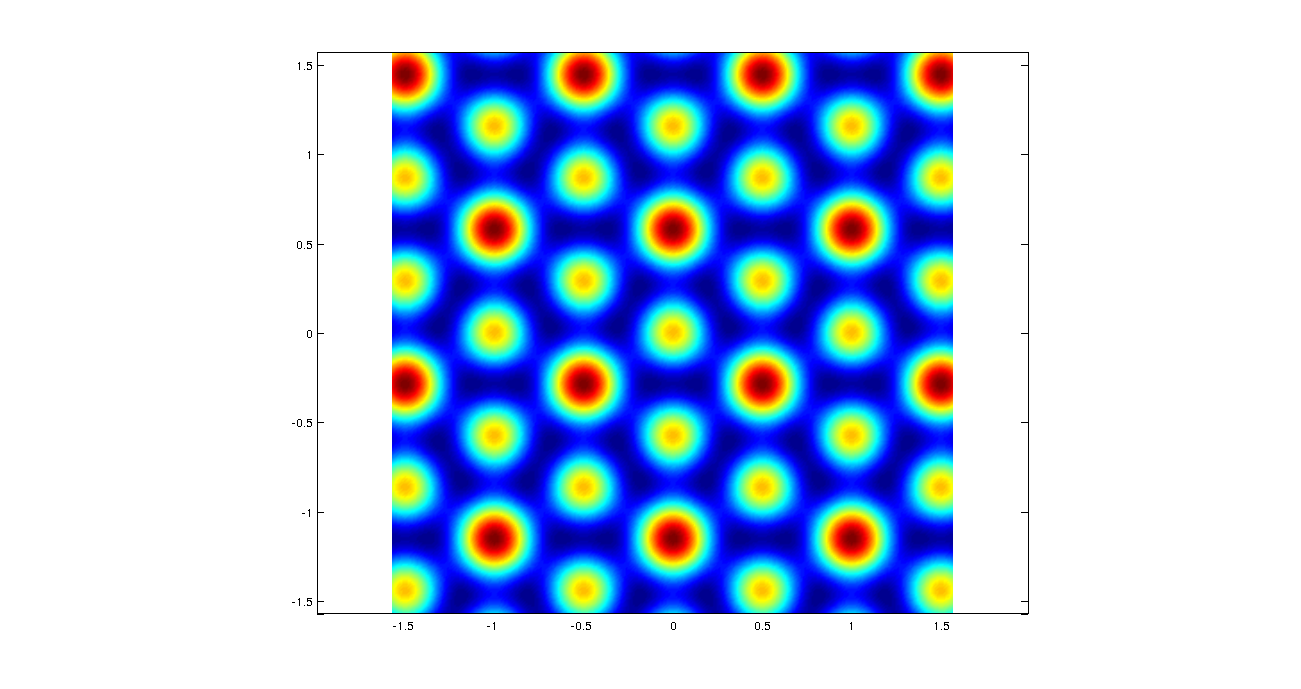}}
(c){\includegraphics[width=0.2\textwidth]{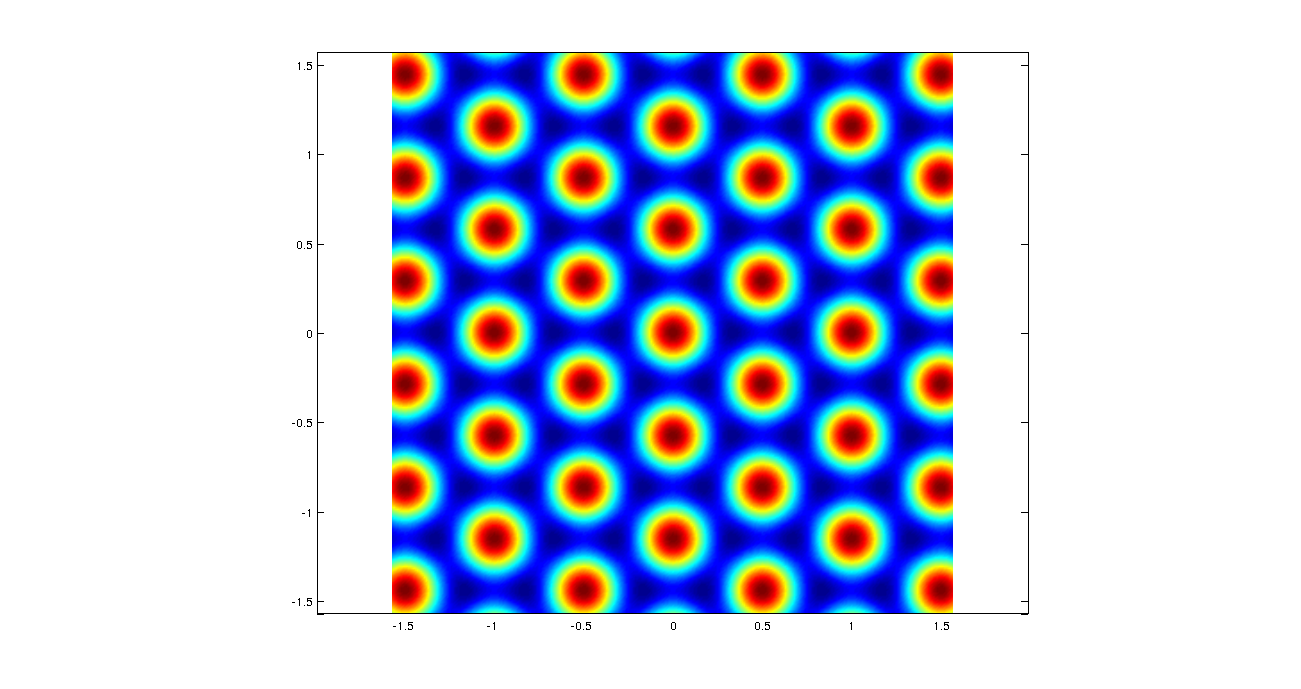}}
(d){\includegraphics[width=0.2\textwidth]{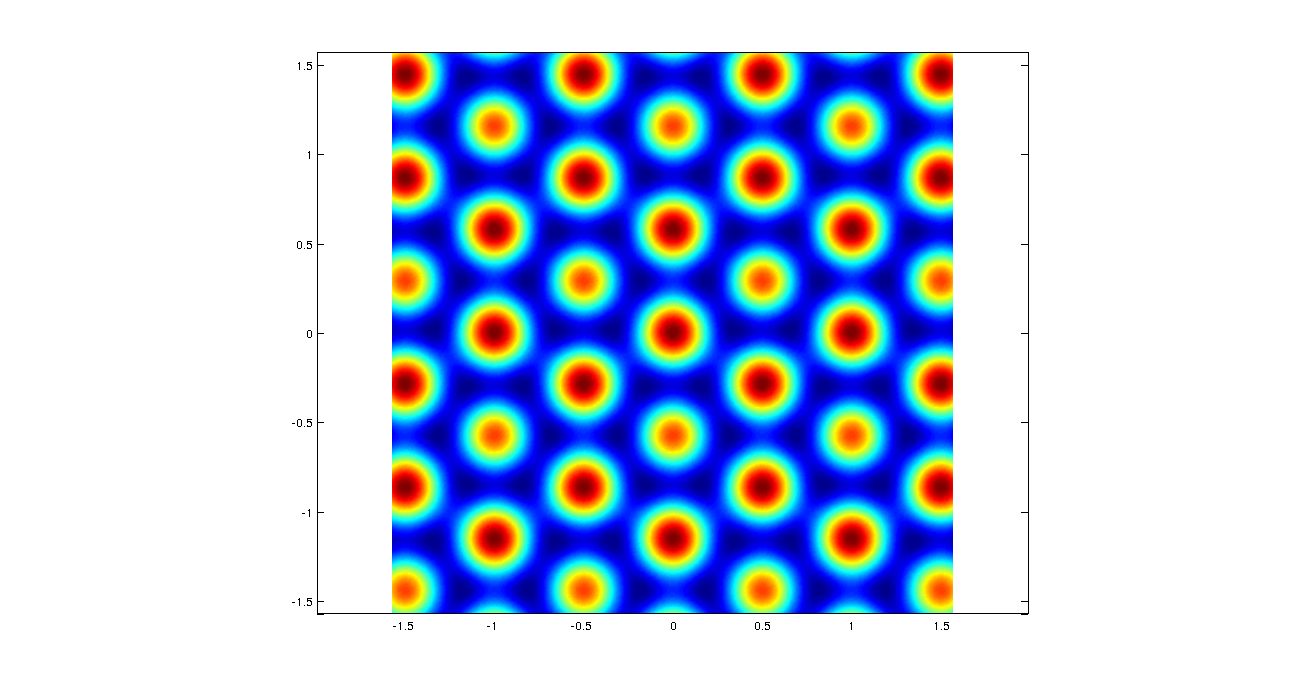}}
\caption{ Projection of pattern $u$ in a 12-dimensional representation with Body Centred Cubic lattice periodicity.  
Contour plots of the integral of $u$ over different depths \(z_{0}\). (a) \(z_{0} = \frac{1}{ 2\sqrt{6}}\). (b) \(z_{0} = \frac{1}{\sqrt{6}}\). (c) \(z_{0} = \frac{3}{2\sqrt{6}}\). (d) \(z_{0} = \frac{2}{\sqrt{6}}\). } 
\end{figure}

As an illustration, consider a systems of generators 
\[l_{1} = (1,0,0), \ l_{2} = (\frac{1}{2},\frac{\sqrt{3}}{2},0), \ l_{3} = (0,0,c) \ \ \ c\neq 0, \ \pm 1\]	
\[r_{1}=(1,0,0), \ r_{2} = (\frac{1}{2},\frac{\sqrt{3}}{2},0),\ r_{3} = (\frac{-1}{2},\frac{\sqrt{3}}{6}, \frac{a}{3}), \ \ \ \ \ a \neq 0\]
for the Hexagonal and Rhombohedral  lattices, respectively. 
A construction similar to that used for the Primitive Cubic lattice may be applied to these two cases,
 but here the parameters  $a$ and $c$ will change the pattern of the projected functions.
 Examples are shown in Figures~\ref{figRhombo} and \ref{figHexa}.
 
 \begin{figure}	\label{figRhombo}
(a){\includegraphics[width=0.2\textwidth]{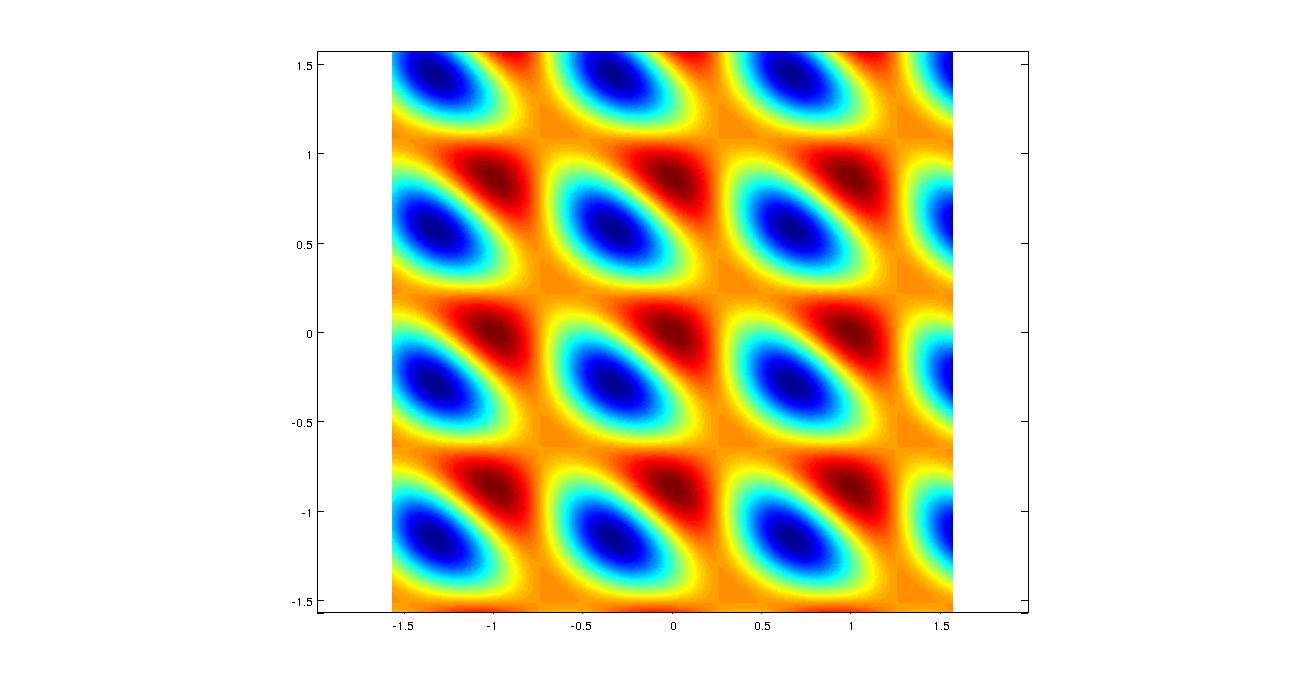}}
(b){\includegraphics[width=0.2\textwidth]{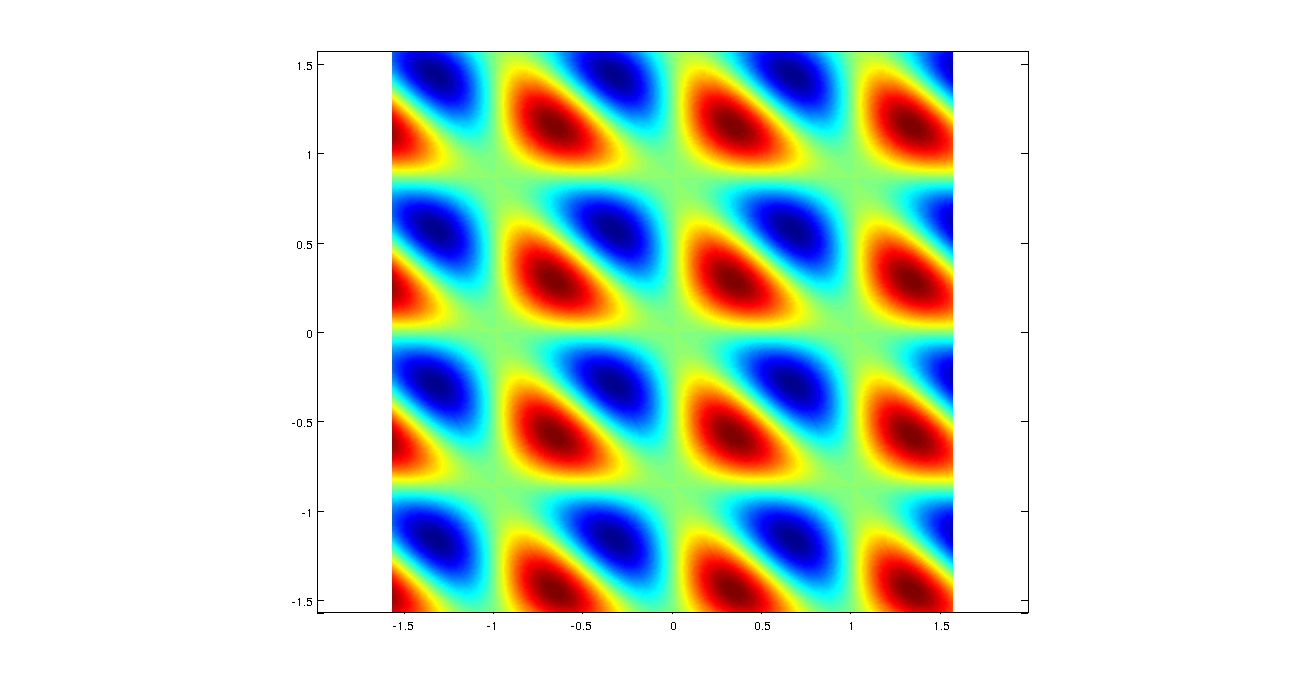}}
(c){\includegraphics[width=0.2\textwidth]{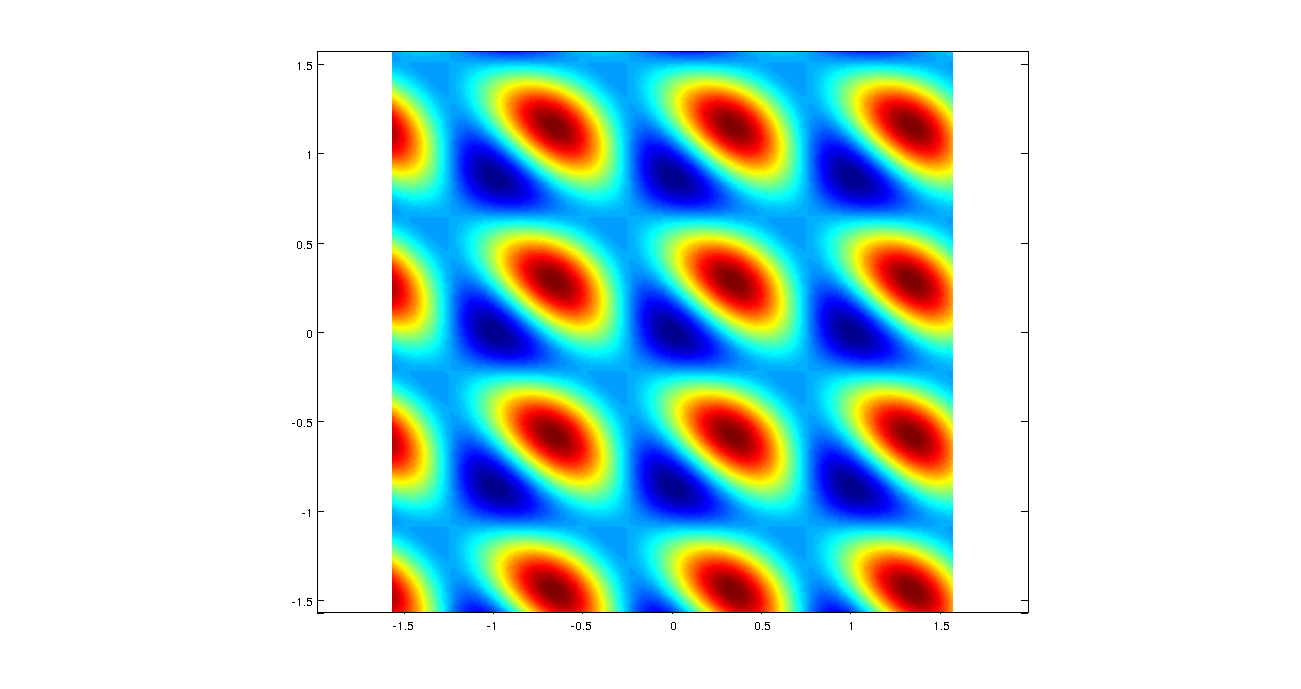}}
(d){\includegraphics[width=0.2\textwidth]{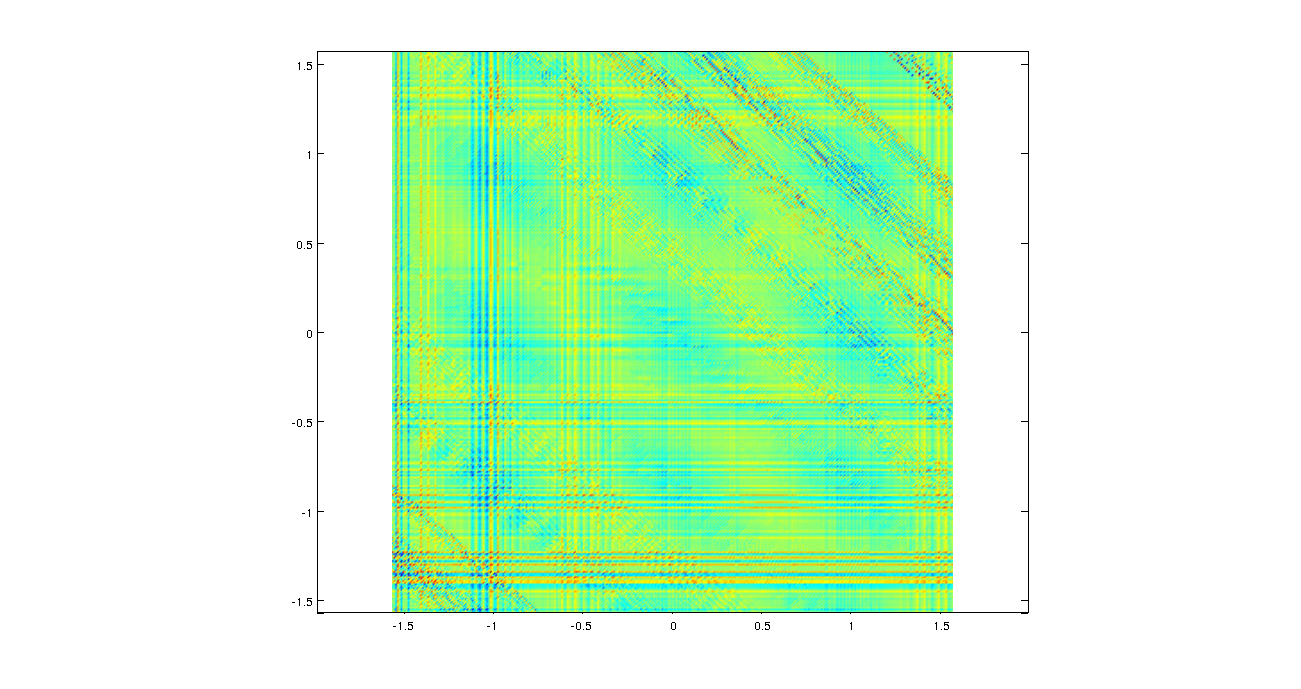}}
\caption{ Projection of pattern $u$  in a 6-dimensional representation with Rhombohedral lattice periodicity.  
Contour plots of the integral of $u$ over different depths \(z_{0}\) with  parameter  $a = 2$ .  (a) \(z_{0} = \frac{2}{6}\). (b) \(z_{0} = \frac{2}{ 3}\). (c) \(z_{0} = \frac{4}{3}\). (d) \(z_{0} = 2\). } 
\end{figure} 
	
\begin{figure}	\label{figHexa}
(a){\includegraphics[width=0.2\textwidth]{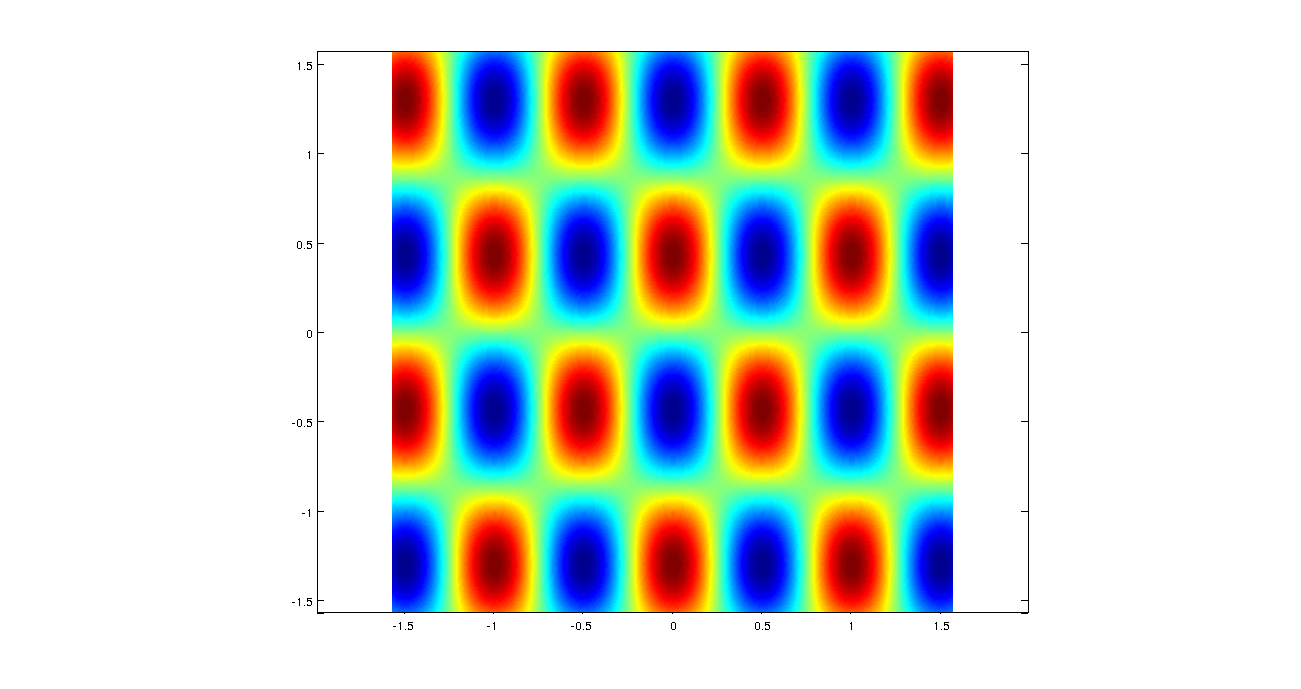}}
(b){\includegraphics[width=0.2\textwidth]{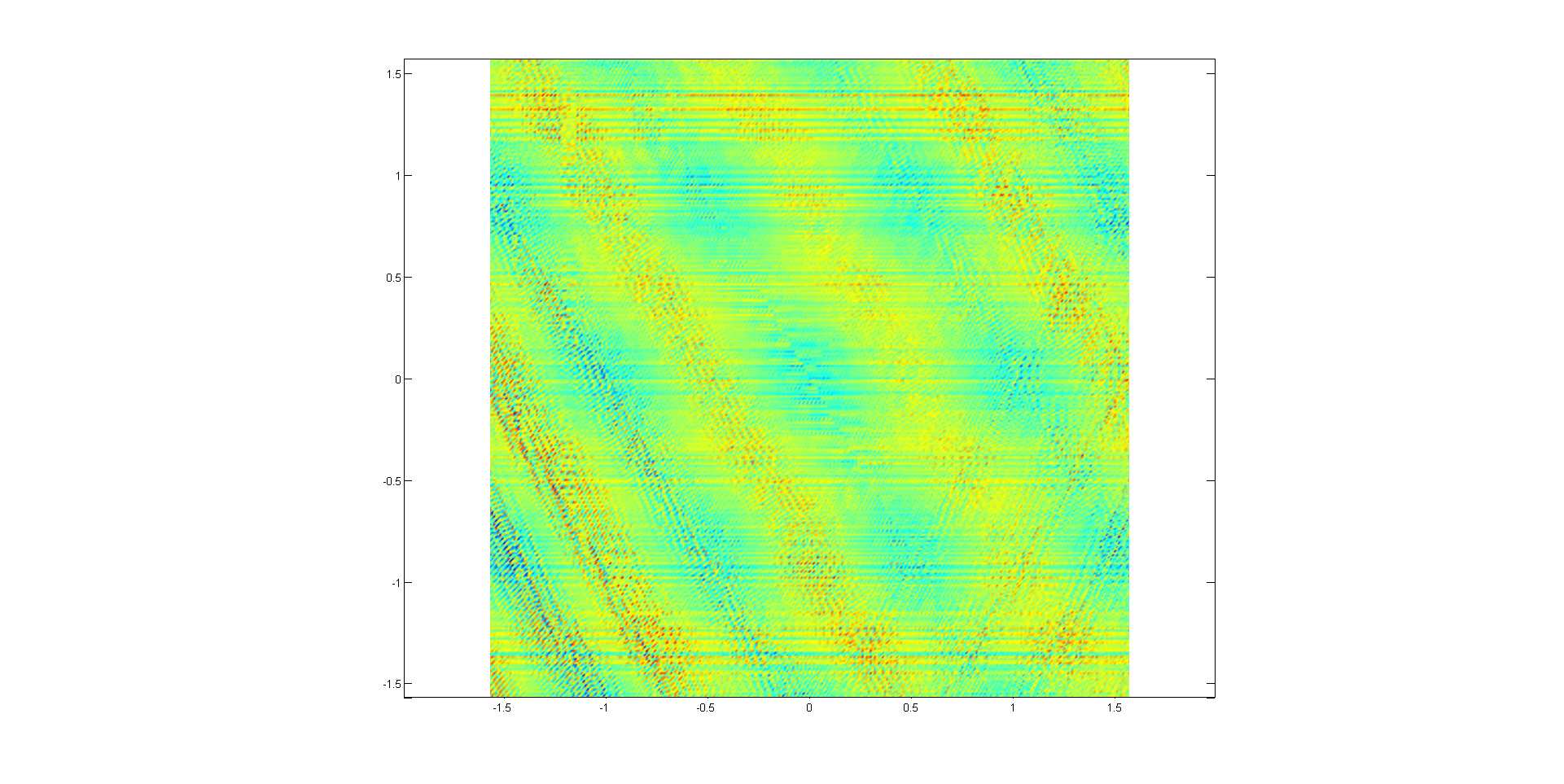}}
\caption{ Projection of pattern $u$ in a 12-dimensional representation with Hexagonal lattice periodicity.  
Contour plots of the integral  of $u$ over different depths \(z_{0}\) with  parameter  $c = 2$ .  (a) \(z_{0} = 1\). (b) \(z_{0} = 2\). } 
\end{figure}

\section{Acknowledgments}

CMUP (UID/MAT/00144/2013) is supported by the Portuguese Government through the Funda\c{c}\~ao para a Ci\^encia e a Tecnologia (FCT) with national (MEC) and European structural funds through the programs FEDER, under the partnership agreement PT2020.
J.F. Oliveira was supported by a grant from the Conselho Nacional de Desenvolvimento Cient\'{i}fico e Tecnol\'{o}gico (CNPq) of Brazil.

\end{document}